\documentclass[11pt]{amsart}
\usepackage{latexsym,amscd,array}
\usepackage{exscale}
\usepackage[centertags]{amsmath}
\usepackage{amssymb,stmaryrd}
\usepackage{amsthm}
\usepackage{dsfont}
\usepackage[all]{xy}
\usepackage{lscape}
\usepackage{pdflscape}
\usepackage{color}
\usepackage{graphicx}
\usepackage{mathdots}

\usepackage{bbm, amscd, amssymb, colonequals, enumerate, mathdots,
  xcolor, tikz-cd, mathtools, amsthm, booktabs, verbatim, amsmath
}

\usepackage{mathrsfs}
\usepackage[neveradjust]{paralist}

\usepackage{lineno}

\definecolor{chianti}{rgb}{0.6,0,0}
\definecolor{meretale}{rgb}{0,0,.6}
\definecolor{leaf}{rgb}{0,.35,0}
\usepackage{hyperref}
\hypersetup{colorlinks=true, citecolor=meretale, urlcolor=leaf, linkcolor=chianti}

\usepackage{thmtools}
\usepackage{thm-restate}

\newtheorem{thm}{Theorem}[section]

\newtheorem{lem}[thm]{Lemma}
\newtheorem{cor}[thm]{Corollary}

\newcommand{\myendsymbol}{\ensuremath{\diamondsuit}}

\declaretheorem[numbered=no,
  title=Theorem,
  qed={},
  refname={theorem,theorems},
  Refname={Theorem,Theorems},
  sharenumber=thm,
]{thmNoNumber}

\declaretheorem[
  style=definition,
  title=Example,
  qed={$\myendsymbol$},
  refname={example,examples},
  Refname={Example,Examples},
  sharenumber=thm,
]{exa}

\declaretheorem[
  style=definition,
  title=Convention,
  qed={$\myendsymbol$},
  sharenumber=thm,
]{cnv}

\declaretheorem[
  style=definition,
  title=Notation,
  qed={$\myendsymbol$},
  sharenumber=thm,
]{ntn}

\declaretheorem[
  style=definition,
  title=Remark,
  qed={$\myendsymbol$},
  sharenumber=thm,
]{rmk}

\declaretheorem[
  style=definition,
  title=Question,
  qed={$\myendsymbol$},
  sharenumber=thm,
]{que}

\def\Pf{\operatorname{Pf}}

\newcommand{\tr}{{\mathrm{T}}}

\newlength{\myl}
\newcommand{\into}{\hookrightarrow}

\renewcommand{\to}{\longrightarrow}

\newcommand{\minus}{\smallsetminus}

\newcommand{\calR}{\mathcal{R}}

\newcommand{\fraka}{{\mathfrak{a}}}

\newcommand{\frakm}{{\mathfrak{m}}}
\newcommand{\frakp}{{\mathfrak{p}}}

\newcommand{\frakP}{{\mathfrak{P}}}

\def\frakP{\mathfrak{P}}

\def\frakq{\mathfrak{q}}

\def\GL{\operatorname{GL}}
\def\Gr{\operatorname{Gr}}
\def\SL{\operatorname{SL}}
\def\Sp{\operatorname{Sp}}

\newcommand{\CC}{\mathbb{C}}

\newcommand{\II}{\mathbb{I}}
\newcommand{\KK}{\mathbb{K}}

\newcommand{\NN}{\mathbb{N}}
\newcommand{\PP}{\mathbb{P}}
\newcommand{\QQ}{\mathbb{Q}}

\newcommand{\ZZ}{\mathbb{Z}}

\newcommand{\YZ}{\ensuremath{(YZ)}}

\def\00{\mathbbm{0}}
\def\11{\mathbbm{1}}

\DeclareMathOperator{\Cl}{\textup{Cl}}

\DeclareMathOperator{\height}{\textup{ht}}

\DeclareMathOperator{\id}{\textup{id}}

\DeclareMathOperator{\ini}{{\textup{in}_{\it B}}}

\DeclareMathOperator{\depth}{\textup{depth}}

\numberwithin{equation}{subsection}

\begin{document}
\title[Classical nullcones]{On the natural nullcones of the symplectic \\ and general linear groups}

\author{Vaibhav Pandey}
\address{Department of Mathematics, Purdue University, West Lafayette, IN~47907, USA}
\email{pandey94@purdue.edu}

\author{Yevgeniya Tarasova}
\address{Department of Mathematics, University of Michigan, Ann Arbor, MI 48109, USA}
\email{ytarasov@umich.edu}

\author{Uli Walther}
\address{Department of Mathematics, Purdue University, West Lafayette, IN~47907, USA}
\email{walther@purdue.edu}

\thanks{
P.V.\ partially supported by an AMS--Simons Travel Grant. 
U.W.\ partially supported by NSF grant DMS-2100288 and by Simons Foundation Collaboration Grant for Mathematicians \#580839. }


\begin{abstract}
Consider a group acting on a polynomial ring $S$ over a field $\KK$ by degree-preserving $\KK$-algebra automorphisms. Several key properties of the invariant ring can be deduced by studying the \textit{nullcone} of the action, that is, the vanishing locus of all non-constant homogeneous invariant polynomials. These properties include the finite generation of the invariant ring and the purity of its embedding in $S$. In this article, we study the nullcones arising from the natural actions of the symplectic and general linear groups.   

For the natural representation of the symplectic group (via copies of the regular representation), the invariant ring is a generic Pfaffian ring. We show that the nullcone of this embedding is $F$-regular in positive characteristic. Independent of characteristic, we give a complete description of the divisor class group of the nullcone and determine precisely when it is Gorenstein.

For the natural representation of the general linear group (via copies
of the regular representation and copies of its dual), the invariant
ring is a generic determinantal ring. The nullcone of this embedding
is typically non-equidimensional; its irreducible components are the varieties of complexes introduced by Buchsbaum and Eisenbud. We show that each of these
irreducible components are $F$-regular in positive characteristic. We also show that the Frobenius splittings of the varieties of complexes may be chosen compatibly so that the nullcone is $F$-pure.  
\end{abstract}
\setcounter{tocdepth}{3}
\maketitle

\section{Introduction}
Let $G$ be a group acting on a standard graded polynomial ring $S$
over a field $\KK$ via degree-preserving $\KK$-algebra
automorphisms. Let $R$ denote the graded subring of $S$ consisting of
all the polynomials which are invariant under the action of each group
element, so $R$ is the invariant ring of this action. Let
$\frakm_R$ denote the homogeneous maximal ideal of $R$. Consider the
ideal of $S$ generated by all non-constant homogeneous polynomials that are invariant under this action. In $1893$, Hilbert
famously showed that if $G$ is a linearly reductive group, then a
minimal set of homogeneous invariant generators of this ideal also generates the ring $R$ as a $\KK$-algebra \cite{Hilbert}. This settled the finite generation of invariant rings for all linearly reductive groups, a result which was previously known only for the special linear group $\SL_2(\CC)$.

The ideal $\frakm_RS$ generated by all homogeneous invariants of
positive degree is often called the \emph{(Hilbert's) nullcone ideal} and
its zero set is referred to as the (Hilbert's) nullcone.  In this article, we call the coordinate ring $S/\frakm_RS$ 
the \emph{nullcone} of the embedding $R \subseteq S$ (see \cite[Chapter
2]{Kemper-Derksen} for a detailed exposition on nullcones). The study of nullcones has been central to classical invariant theory and has more recently found renewed interest in algebraic geometry, representation theory, and commutative algebra \cite{Hesselink, Wang, Kraft-Wallach, Kim, Kraft-Schwarz, Lorincz, HJPS}. In a series of exciting new developments, efficient algorithms have been developed to test if a given point lies in the nullcone (the \emph{nullcone membership problem}) and these algorithms have led to intriguing connections with computational complexity theory. We point the reader to \cite{Makam-Wigderson} and the references therein.

As a first example, note that the nullcone for the action of a finite group is an Artinian ring. For a more illuminating example, let $Y$ be a $d\times n$ matrix of indeterminates over~$\KK$, and
set~$S$ to be the polynomial ring $\KK[Y]$. Let $R =
\KK[\{\Delta\}]$ be the $\KK$-algebra generated by the size
$d$ minors $\{\Delta\}$ of~$Y$.
Then~$R$ is the coordinate ring for the Pl\"ucker embedding of the
Grassmannian of~$d$-dimensional subspaces of
an~$n$-dimensional vector space. The special linear
group~$\SL_d(\KK)$ acts $\KK$-linearly on $S$ via the action
\[
M\colon Y\mapsto MY\qquad\text{ for }\ M\in\SL_d(\KK).
\]
When $\KK$ is an infinite field, the invariant ring is precisely~$R$, see \cite{Igusa} or \cite[\S 2]{DeConciniProcesi76}. Clearly, the nullcone of the natural embedding
\[R = \KK[\{\Delta\}] \subseteq \KK[Y] = S\]
is the determinantal ring $S/I_d(Y)$ defined by the maximal minors of a generic matrix.

When $\KK$ has characteristic zero, the group $\SL_d(\KK)$ is linearly
reductive. It follows that $R$ is a direct summand of $S$ as an
$R$-module. In contrast, when $\KK$ has positive characteristic, it
was recently shown in \cite[Theorem 1.1]{HJPS} that the above natural
embedding typically does \emph{not} split. This non-splitting is due
to the Cohen--Macaulay property of the nullcone, when combined with
the flatness of the Frobenius map on the ambient polynomial ring $S$.

Let us review the key algebro-geometric properties of the above nullcone: The Cohen--Macaulay property of determinantal rings of maximal minors was proved using the Eagon--Northcott resolution in \cite{Eagon-Northcott} and later for minors of all sizes by the technique of principal radical systems in \cite{Hochster-Eagon}. The divisor class group of determinantal rings is the infinite cyclic group by \cite{Bruns} and these rings  are Gorenstein precisely when the matrix is square by \cite[Theorem 5.4.6]{Svanes}. The initial ideals of generic determinantal ideals are squarefree since the minors form a Gr\"obner basis by \cite{Sturmfels} and \cite{CGG}. Determinantal rings are $F$-regular in positive characteristic by \cite[\S 7]{HH94} (see also \cite[Theorem 4.4]{Pandey-Tarasova}) and so they have rational singularities in characteristic zero by \cite[Theorem 4.3]{Smith}.

The objective of this paper is to discuss the corresponding results for the nullcones arising from the natural actions of the symplectic group and the general linear group. We now describe these group actions, and our main results.

Let $Y$ be a $2t\times n$ matrix of indeterminates for positive integers $t$ and $n$. Set~$S$ to be the polynomial ring $\KK[Y]$ and let
\begin{equation*}
\Omega\colonequals\begin{pmatrix} 0 & \II \\ -\II & 0 \end{pmatrix}
\end{equation*}
be the size $2t$ standard symplectic block matrix, where $\II$ is the size $t$ identity matrix. The $\KK$-algebra $R\colonequals \KK[Y^\tr \Omega Y]$ is canonically isomorphic to $\KK[X]/\Pf_{2t+2}(X)$, where~$X$ is an $n\times n$ alternating matrix of indeterminates, and $\Pf_{2t+2}(X)$ the ideal generated by its principal size~$2t+2$ Pfaffians. The symplectic group
\[
\Sp_{2t}(\KK)\colonequals\{M\in\GL_{2t}(\KK)\ |\ M^\tr\Omega M=\Omega\}
\]
acts $\KK$-linearly on $S$ via the action
\[
M\colon Y\mapsto MY\qquad\text{ for }\ M\in\Sp_{2t}(\KK).
\]
The invariant ring is precisely $R$ whenever $\KK$ is infinite, see \cite[\S 6]{DeConciniProcesi76} or \cite[Theorem 5.1]{Hashimoto05}. Notice that the ideal \[\frakP (Y):= (Y^\tr\Omega Y)S\] generated by the entries of the alternating matrix $Y^\tr\Omega Y$ is the nullcone ideal for this group action. In this article, we refer to the nullcone 
of the natural embedding 
\[
R = \KK[Y^\tr \Omega Y] \subseteq  \KK[Y] = S. 
\]
as the (natural) \emph{symplectic nullcone} $\KK[Y]/\frakP(Y)$.

When $\KK$ has characteristic zero, the group $\Sp_{2t}(\KK)$ is
linearly reductive and thus $R$ is a direct summand of $S$ as an
$R$-module. When $\KK$ has positive characteristic, this embedding
typically does \emph{not} split by \cite[Theorem 1.1]{HJPS}.

As in the
case of the Pl\"ucker embedding of Grassmannians, this non-splitting is due to the
Cohen--Macaulay property of the symplectic nullcone \cite[Theorem 1.2]{HJPS}, in conjunction with the flatness of
Frobenius. Earlier, the irreducibility and normality of the 
nullcone was proved in \cite[Theorem 9.1 (3)]{Kraft-Schwarz}. We prove
a much stronger result by establishing the $F$-regularity of the symplectic nullcone. We also show that the  nullcone ideal has a squarefree initial ideal and is hence amenable to Gr\"obner degeneration techniques, for example, as discussed in \cite{Conca-Varbaro}.

\begin{thmNoNumber}[Theorem \ref{thm-PfaffianNullcone-Fregular}, Corollary \ref{Cor:PfaffianNullconeSqfree}]
  Let $Y$ be a matrix of indeterminates of size $2t \times n$ for
  positive integers $t$ and $n$.
  Let $\KK$ be a field and set
  $S:=\KK[Y]$. 
  
\begin{enumerate}[\quad\rm(1)]

        \item If $\KK$ is an $F$-finite field of positive
          characteristic, the ring $S/\frakP(Y)$
          is strongly $F$-regular.
        
        \item If $\KK$ has characteristic zero, $S/\frakP(Y)$ has rational singularities.

        \item For a suitable term order, the initial ideal of $\frakP(Y)$ is squarefree. 
    \end{enumerate}    
\end{thmNoNumber}

We note that in \cite[Proposition 4.7]{Lorincz23}, L\H{o}rincz also proves the $F$-regularity of the symplectic nullcone by realizing it as the image of a collapsing map of a homogeneous vector bundle (see also \cite[Proposition 5.4]{Lorincz23}). Independent of characteristic, we give a complete description of the divisor class group of the symplectic nullcone and determine precisely when it is Gorenstein.

\begin{thmNoNumber}[Theorem \ref{TheoremClassGroup}]
    Let $Y$ be a $2t \times n$ matrix of indeterminates for positive integers $t$ and $n$, and $\KK$ be any field. Let $R_{t,n}$ denote the symplectic nullcone $\KK[Y]/\frakP(Y)$. 
    \begin{enumerate}[\quad\rm(1)]
    
        \item If $n\geq t+1$, the divisor class group of $R_{t,n}$ is the integers. Otherwise, $R_{t,n}$ is a unique factorization domain.

         \item Let $Y|_t$ denote the submatrix of $Y$ consisting of the first $t$ columns. If $n \geq t+1$, the ideal $\frakp := I_{t}(Y|_t)$ of $R_{t,n}$ is prime of height one and generates the divisor class group. Furthermore, the canonical class of $R_{t,n}$ is given by $\frakp^{(n-t-1)}$.  

        \item $R_{t,n}$ is a Gorenstein ring if and only if $n \leq t+1$.    
        \end{enumerate}   
\end{thmNoNumber}

We now discuss the embedding defined by the natural action of the general linear group.
For positive integers $m$, $t$, and $n$, let $Y$ and $Z$ be $m\times t$ and $t\times n$ matrices of indeterminates respectively. Set $S$ to be the polynomial ring $\KK[Y,Z]$, and take $R$ to be the $\KK$-subalgebra generated by the entries of the product matrix $YZ$. Then $R$ is canonically isomorphic to the determinantal ring $\KK[X]/I_{t+1}(X)$, where~$X$ is an $m\times n$ matrix of indeterminates and $I_{t+1}(X)$ denotes the ideal generated by the size~$t+1$ minors of $X$. The general linear group $\GL_t(\KK)$ acts~$\KK$-linearly on $S$ via
\[
M\colon\begin{cases} Y & \mapsto YM^{-1}\\ Z & \mapsto MZ\end{cases}
\]
where $M\in\GL_t(\KK)$. When the field $\KK$ is infinite, $R$ is precisely the ring of invariants, see \cite[\S 3]{DeConciniProcesi76} or \cite[Theorem 4.1]{Hashimoto05}.
The nullcone of the natural embedding
\[R = \KK[YZ] \subseteq \KK[Y,Z] = S\] is the typically
non-equidimensional reduced ring $\KK[Y,Z]/(YZ)$ defined by the entries of the matrix $YZ$ (see \cite{Musili-Sheshadri} or \cite[Theorem 4.1]{Mehta-Trivedi2} for the reducedness of the nullcone). In
this article, we refer to this ring as the (natural) \emph{generic nullcone}. The minimal primes of the generic
nullcone are the ideals
\[
\frakp_{r,s}(Y,Z):= I_{r+1}(Y)+I_{s+1}(Z)+\YZ
\]
of $S$, for $r+s = t$; see $\S$\ref{sec-splitting-generalities} for details. These are precisely the ideals defining the
varieties of exact complexes, as introduced by
Buchsbaum and Eisenbud in \cite{Buchsbaum-Eisenbud}.

When $\KK$ has characteristic zero, the group $\GL_t(\KK)$ is linearly
reductive and thus the determinantal ring $R$ splits from $S$ as an
$R$-module. When $\KK$ has positive characteristic, this embedding
typically does \emph{not} split by \cite[Theorem 1.1]{HJPS}. This is
due to the flatness of the Frobenius together with the Cohen--Macaulay
property of the varieties of complexes which was proved in
\cite[Theorem 6.2]{Huneke} using principal radical systems (see also \cite{Musili-Sheshadri, DeConcini-Lakshmibai}). The
divisor class groups of the varieties of complexes are free abelian
groups of finite ranks by \cite[Theorem 3.1]{Bruns2} and independently
by \cite[Theorem 1.1]{Yoshino}. These rings are Gorenstein under
certain symmetry conditions on the sizes of the minors involved by
\cite[Theorem 4.3]{Bruns2} and by \cite[Theorem 1.2]{Yoshino}.

Kempf showed in \cite{Kempf2} that the varieties of complexes have rational singularities, using \cite{Kempf1}. In positive characteristic, they are $F$-rational relative to the resolution of Kempf, and they are also $F$-split by \cite{MehtaTrivedi}. We prove:

\begin{thmNoNumber}[Theorems \ref{thm-Complex-Fregular}, \ref{thm-length-2-F-purity}]
  Let $Y$ and $Z$ be matrices of indeterminates of sizes $m \times t$
  and $t \times n$ respectively for positive integers $m$, $t$, and $n$. Let $\KK$ be an $F$-finite field of
  positive characteristic; set $S :=\KK[Y,Z]$ and 
  suppose that $r$ and $s$ are non-negative integers with  $r+s \leq t$. 
    \begin{enumerate}[\quad\rm(1)]
    \item The variety of complexes $S/\frakp_{r,s}(Y,Z)$ is strongly
      $F$-regular.
    \item If $t\le \min(m,n)$, the splittings of the Frobenius map on $S/\frakp_{r,s}(Y,Z)$ 
      may be chosen compatibly. 
      \item For any triple $(m,n,t)$, the generic nullcone $S/(YZ)S$ is an $F$-pure ring.
    \end{enumerate}
\end{thmNoNumber}
We remark that the above theorem is not new:
 L\H{o}rincz recently proved the $F$-regularity of the varieties of complexes using the collapsing of homogeneous vector bundles \cite[Corollary 4.2]{Lorincz}, while the fact that varieties of complexes are compatibly $F$-split follows from Kinser--Rachogt's study of the Frobenius splittings of type A quiver varieties \cite[Corollary 1.4 (ii)]{Rajchgot-Kinser}.  
On the other hand, the techniques and constructions involved in our proof are novel and yield an interesting consequence: the generic nullcone and the varieties of complexes have squarefree initial ideals (Corollary \ref{cor:VoCsqfreeInitial}).

The critical step in establishing the $F$-regularity of the symplectic nullcone and the varieties of complexes is the construction of `appropriate' Frobenius splittings. To this end, we create certain delicate and non-obvious monomial orderings in each case; these orders select special lead terms for the generators of the respective nullcone ideal in a characteristic-free manner. We show that the nullcone ideal of interest has a squarefree initial ideal with respect to these monomial orders. The said monomial orders are explained with examples in $\S$\ref{Subsection
  PfaffianMonomialOrder} and $\S$\ref{Subsection VoCMonomialOrder}.

\section{Testing for $F$-purity and $F$-regularity}

We briefly recall some details on the types of singularities in positive characteristic that we discuss in this paper, as well as  certain tools to test for them. For a detailed exposition on Frobenius singularities, we refer the reader to the lecture notes \cite{Ma-Polstra} of Ma and Polstra.  

Let $R$ be a reduced Noetherian ring of positive prime characteristic $p$. The letter $e$ denotes a variable nonnegative integer, and $q=p^e$ the $e$-th power. Let $R^{1/q}$ denote the ring obtained by adjoining all $q$-th roots of elements of $R$. The inclusion $R\into R^{1/q}$ can then be identified with the $e$-fold Frobenius endomorphism $F^e\colon R\to R$.

The ring $R$ is \emph{$F$-finite} if it is a finitely generated
$R$-module via the action of the Frobenius endomorphism $F: R \to R$,
or equivalently, if $R^{1/p}$ is finitely generated as an
$R$-module. A finitely generated algebra over a field $\KK$ is
$F$-finite if and only if $\KK^{1/p}$ is a finite field extension of
$\KK$. For an ideal $I=(z_1, z_2, \ldots, z_t)$ of $R$, the symbol
$I^{[q]}$ denotes the ideal $F(I)R = (z_1^q,z_2^q,\ldots,z_t^q)$ of
$R$.

The ring $R$ is \emph{$F$-pure} if $F$ is a pure homomorphism, i.e., if the map $R \otimes_R M \to R^{1/p}\otimes_R M$ induced by the
inclusion of $R$ in $R^{1/p}$ is injective for each $R$-module $M$. We
say that $R$ is \emph{$F$-split} if $F$ is a split
monomorphism. Clearly, any
$F$-split ring is $F$-pure; furthermore, an algebra over an $F$-finite field is $F$-pure if and
only if it is $F$-split. 

An $F$-finite ring $R$ is \emph{strongly $F$-regular} if for each $c$
in $R$ not in any of its minimal primes, there exists an integer $q$
such that the $R$-linear inclusion $R \to R^{1/q}$ sending $1$ to
$c^{1/q}$ splits as a map of $R$-modules. If the $F$-finite ring $R$
is $\NN$-graded, strong $F$-regularity coincides with other similar
notions like $F$-regularity and weak $F$-regularity.

Before moving forward, we fix some general hypotheses.
\begin{cnv}
Unless stated otherwise,  all rings in this paper are standard graded
algebras over a field. When the field is of positive characteristic,
it is assumed to be $F$-finite.

Since all rings in this paper are $F$-finite,
we may use the words $F$-pure and $F$-split interchangeably.
Similarly, since all rings
in this paper are standard graded algebras over $F$-finite fields, we
sometimes loosely refer to a strongly $F$-regular ring simply as an
$F$-regular ring.
\end{cnv}

The following criterion of Fedder is useful for testing $F$-purity in the homogeneous setting:

\begin{thm}\label{theoremFedder} \cite[Theorem 1.12]{Fedder}
Let $S = \KK[x_1, \ldots, x_n]$ be an $\mathbb{N}$-graded polynomial ring over a field $\KK$ of positive characteristic. Let $I$ be a homogeneous ideal of $S$ and let $\frakm_S$ denote the homogeneous maximal ideal of $S$. Then $S/I$ is $F$-pure if and only if the ideal $I^{[p]}:I$ is not contained in ${\frakm_S}^{[p]}$.
\end{thm}

The following criterion of Glassbrenner is useful for testing $F$-regularity in the graded setting:

\begin{thm}\label{theoremGlassbrenner} \cite[Theorem 3.1]{Glassbrenner}
  Let $\KK$ be a field of positive characteristic $p$ and
  set $S=\KK[x_1,\ldots,x_n]$, with homogeneous maximal ideal
  $\frakm_S$. Suppose that $R=S/I$ is a finitely generated
  $\NN$-graded domain. Then $R$ is strongly $F$-regular if and only if
    \begin{itemize}
    \item there exists a homogeneous element $s$ of $S$, not in $I$,
      for which the ring $R[1/s]$ is strongly $F$-regular, and
    \item the ideal $s(I^{[p]}:I)$ is not contained in  $\frakm_S^{[p]}$.\qed
    \end{itemize}
\end{thm}

The next two facts are helpful in proving the $F$-purity of a given ring. We shall use the following as an ingredient in establishing the $F$-regularity of the symplectic nullcone:

\begin{cor}\label{corMain}\cite[Corollary 3.3]{Pandey-Tarasova}
  Let $S$ be a polynomial ring over a field of positive prime
  characteristic $p$ and let $I$ be an equidimensional ideal of
  $S$. If $\fraka \subsetneq I$ is an ideal generated by a regular
  sequence of length equal to the height of $I$, then the ideal
  $\fraka^{[p]}:\fraka$ is contained in $I^{[p]}:I$. In particular, if
  the ring $S/\fraka$ is $F$-pure, then so is $S/I$.
\end{cor}

\begin{proof}
     Let $J$ be the ideal $\fraka : I$ of $S$. We have the following chain of containments of ideals in $S$:
\begin{eqnarray*} 
        \fraka^{[p]}:\fraka &\subseteq& \fraka^{[p]}: IJ \\
        &=&  (\fraka^{[p]}: J): I \\
        &\subseteq&  (\fraka^{[p]}: J^{[p]}): I \\
        &=&  (\fraka: J)^{[p]}: I \\
        &=&  I^{[p]}: I,
\end{eqnarray*}
where the second-to-last equality follows from the flatness of the
Frobenius map on the regular ring $S$ and the last equality follows
from the symmetry of links since $I$ is an equidimensional ideal; see
\cite{PS74}. The result is now immediate from Theorem
\ref{theoremFedder}.
\end{proof}

\begin{cor}\label{CorollarySymbolic}
Let $S$ be a polynomial ring over a field of positive characteristic $p$ and let $\frakm_S$ denote its homogeneous maximal ideal. Let $I$ be a height $h$ prime ideal of $S$. Then the symbolic power $I^{(h(p-1))}$ is contained in $I^{[p]}: I$. In particular, if $I^{(h(p-1))}$ is not contained in $\frakm_S^{[p]}$, then $S/I$ is $F$-pure.  
\end{cor}

\begin{proof} By the flatness of the Frobenius map on $S$, the set of associated prime ideals of $S/I^{[p]}$ equals that of $S/I$. So the containment of ideals
\[
I\cdot I^{(h(p-1))}\subseteq I^{[p]} 
\] 
may be verified locally. In the regular ring $(S_I, IS_I)$, the maximal ideal $IS_I$ is generated by $h$ elements. Recall that the ordinary and symbolic powers of ideals primary to the maximal ideal are equal. Notice that the containment
\[
  I^{h(p-1)+1}S_I \subseteq I^{[p]}S_I
\]
immediately follows from the pigeonhole principle. In fact, this
result holds more generally when $I$ is a radical ideal and $h$ is the
maximum of the heights of the associated prime ideals of
$I$. 
\end{proof}

Corollary \ref{CorollarySymbolic} is especially useful when we
know the primary decomposition of the powers of the ideal of
interest. We shall need it in proving the compatible $F$-splittings of
varieties of complexes by making use of the primary decomposition of
the powers of determinantal ideals. We note that Corollary \ref{CorollarySymbolic} appears implicitly in the proof
that Hankel determinantal rings are $F$-pure \cite[Theorem 4.1]{CMSV}. In fact, if an ideal $I$ satisfies the conclusion of Corollary \ref{CorollarySymbolic}, then its symbolic powers $\{I^{(n)}\}_{n \geq 0}$ define an $F$-split filtration, i.e., $I$ is a \emph{symbolic $F$-split} ideal---a notion which is stronger than $F$-split, see \cite[Corollary
5.10, Example 5.13]{dSMNB21}.

The following theorem is first proved in positive characteristic by an application of Theorem \ref{theoremFedder}; it is then proved in characteristic zero by using reduction mod $p$ techniques. We use this result in proving that several ideals considered in this paper have squarefree initial ideals (in any characteristic) with respect to the monomial orders $<_B$ constructed in $\S$\ref{Subsection
PfaffianMonomialOrder} and $\S$\ref{Subsection VoCMonomialOrder}. 

\begin{thm} \cite[Theorem 3.13]{Varbaro-Koley} \label{theorem:sqfreeinitial}
    Let $S$ be a polynomial ring over a field, $I$ a radical ideal and $<_B$ a monomial order in $S$. Let $h$ be the maximum of the heights of the associated primes of $I$. If the initial ideal $\ini(I^{(h)})$ contains a squarefree monomial, then the ideal $\ini(I)$ is a squarefree. 
\end{thm}

\section{The symplectic nullcone is strongly $F$-regular}

The aim of this section is to establish the $F$-regularity of the symplectic nullcone. We begin with recalling some known facts.

\subsection{Generalities}
Let $Y$ be a $2t\times n$ matrix of indeterminates and set~$S\colonequals \KK[Y]$. Let
\[
\Omega\colonequals\begin{pmatrix} 0 & \II \\ -\II & 0 \end{pmatrix}
\]
be the $2t\times 2t$ standard symplectic block matrix, where $\II$ is the
size $t$ identity matrix. There is a natural $\KK$-algebra isomorphism 
\[
R\colonequals \KK[Y^\tr \Omega Y]
\cong \KK[X]/\Pf_{2t+2}(X),
\]
where~$X$ is an $n\times n$ alternating matrix of indeterminates, and $\Pf_{2t+2}(X)$ the ideal
generated by its principal Pfaffians of size~$2t+2$. This isomorphism is induced by mapping the entries of the matrix $X$ to the corresponding entries of the alternating matrix $Y^\tr \Omega Y$. The symplectic
group
\[
\Sp_{2t}(\KK)\colonequals\{M\in\GL_{2t}(\KK)\ |\ M^\tr\Omega M=\Omega\}
\]
acts $\KK$-linearly on $S$ via  
\[
M\colon Y\mapsto MY\qquad\text{ for }\ M\in\Sp_{2t}(\KK).
\]
The invariant ring is precisely $R$ whenever $\KK$ is infinite,
see~\cite[\S6]{DeConciniProcesi76}
or~\cite[Theorem~5.1]{Hashimoto05}. The nullcone ideal of this group action is the ideal
\[
\frakP = \frakP(Y)\colonequals (Y^\tr\Omega Y)S,
\]
generated by the entries of the matrix $Y^\tr \Omega Y$ in
the polynomial ring $S$. The nullcone for this action of
$\Sp_{2t}(\KK)$ on $S$, the (natural) \emph{symplectic nullcone}, is the
ring $S/\frakP$. The symplectivc nullcone is a Cohen--Macaulay normal
domain with
\begin{eqnarray*}\label{eqn-S/p-dim}
\dim S/\frakP&=&\begin{cases}
2nt -\displaystyle{\binom{n}{2}}& \text{if }\ n\le t+1,\\
nt+\displaystyle{\binom{t+1}{2}}& \text{if }\ n\ge t,\\
\end{cases}
\end{eqnarray*}
according to \cite[Theorem 6.8]{HJPS} and \cite[Theorem 9.1(3)]{Kraft-Schwarz}. It is a complete intersection ring precisely when $n \leq t+1$ by \cite[Theorem 6.3]{HJPS}. The standard monomials for the symplectic nullcone are studied in \cite[\S 4.2]{Kim}. 

\begin{rmk}\label{RmkGeneratorsofNullconeIdeal} 
Let $Y$ be a size $2t\times n$ matrix of indeterminates over a field $\KK$. Notice that when $t=1$, we have

\begin{alignat*}2
Y^\tr\Omega Y\ &=\ \begin{pmatrix}
y_{1,1} & y_{2,1}\\
\vdots&\vdots\\
y_{1,n} & y_{2,n}
\end{pmatrix}
\begin{pmatrix}
0 & 1\\
-1 & 0
\end{pmatrix}
\begin{pmatrix}
y_{1,1} & \cdots & y_{1,n}\\
y_{2,1} & \cdots & y_{2,n}
\end{pmatrix}
=\ \begin{pmatrix}
0 & \Delta_{1,2} & \Delta_{1,3} & \hdots & \Delta_{1,n}\\
-\Delta_{1,2} & 0 & \Delta_{2,3} & \hdots & \Delta_{2,n}\\
-\Delta_{1,3} & -\Delta_{2,3} & 0 & \hdots & \Delta_{3,n}\\
\vdots & \vdots & \vdots & \ddots & \vdots\\
-\Delta_{1,n} & -\Delta_{2,n} & -\Delta_{3,n} & \hdots & 0
\end{pmatrix}.
\end{alignat*}
In particular, $Y^\tr\Omega Y$ is an alternating matrix where, for $i<j$, the matrix entry $(Y^\tr\Omega Y)_{ij}$ is
\[
\Delta_{i,j}\colonequals y_{1,i}y_{2,j}-y_{1,j}y_{2,i}.
\]
It follows that $\frakP$ coincides with the determinantal ideal $I_2(Y)$. So, $\frakP$ has height $n-1$ and defines an $F$-regular ring $\KK[Y]/\frakP$ \cite[\S 7]{HH94}. Note that the ring $\KK[Y^\tr\Omega Y]$ is the homogeneous coordinate ring of the Grassmannian $\Gr_\KK(2,n)$ under the Pl\"ucker embedding into $\PP_\KK^{\binom{n}{2}-1}$.

More generally, for $t\ge 1$, the ring $\KK[Y^\tr\Omega Y]$ is the homogeneous coordinate ring of the order $t-1$ secant variety $\Gr_\KK^{t-1}(2,n)$,  the closure of the union of linear spaces spanned by $t$ points of $\Gr_\KK(2,n)$ under the Pl\"ucker embedding (see \cite{CGG2}, \cite[\S 6.2]{HJPS}). Moreover, for $1\le i<j\le n$, the entry $(Y^\tr\Omega Y)_{i,j}$ equals $B(v_i,v_j)$, where $v_i$ and $v_j$ are the $i$-th and $j$-th columns of $Y$, and $B$ is the (nondegenerate) symplectic form 
\[
(v_1,v_2)\mapsto v_1^\tr\Omega v_2. 
\] 
It follows that the generators of $\frakP$ are sums of size two minors given by
\[
(Y^\tr\Omega Y)_{i,j}\ =\ (y_{1,i}y_{t+1,j}-y_{1,j}y_{t+1,i})\ +\ \cdots\ +\ (y_{t,i}y_{2t,j}-y_{t,j}y_{2t,i}) 
\quad \text{for}\; 1 \leq i<j\leq n. \qedhere \]
\end{rmk}

\subsection{The localization property}
We show that the symplectic nullcone has a localization property roughly analogous to that of the generic determinantal ring, as outlined in \cite[Proposition 2.4]{BrunsVetter}.
\begin{lem}
\label{lemma:matrix:invert}
Let $Y=(y_{i,j})$ be a $2t\times n$ matrix of indeterminates where $t\geq 2$ and $n\geq 2$; set $S\colonequals\ZZ[Y]$. Then there exists a $(2t-2)\times(n-1)$ matrix $Z$ with entries in $S[\frac{1}{y_{1,1}}]$, and elements $f_2,\dots,f_n$ in $S[\frac{1}{y_{1,1}}]$ such that:
\begin{enumerate}[\quad\rm(1)]
\item The entries of $Z$ and the elements $f_2,\ldots,f_n, y_{1,1},\ldots,y_{1,n}, y_{2,1},\dots,y_{2t,1}$ taken together are algebraically independent over $\ZZ$;
\item Along with $y_{1,1}^{-1}$, the above elements generate $S[\frac{1}{y_{1,1}}]$ as a $\ZZ$-algebra;
\item With $S':=\ZZ[Z]$, the ideal $\frakP(Y)S[\frac{1}{y_{1,1}}]$ equals $\frakP(Z)S[\frac{1}{y_{1,1}}]+(f_2,\dots,f_n)S[\frac{1}{y_{1,1}}]$, and we have an isomorphism

\[\frac{S}{\frakP(Y)}[\frac{1}{y_{1,1}}] \cong \frac{S'}{\frakP(Z)}[y_{1,1},\dots,y_{1,n}, y_{2,1},\dots,y_{2t,1},\frac{1}{y_{1,1}}]. \]
\end{enumerate}
\end{lem}

\begin{proof}
Let us map the entries of the matrix $Y$ to the corresponding entries of $YM$, where $M$ is a matrix with $n$ rows. Clearly, the ideal $(YM)$ generated by the entries of the matrix $YM$ is contained in $(Y)$, hence the ideal $\frakP(YM)$ is contained in $\frakP(Y)$. It follows that if $M$ is invertible in $S$, then the ideals $\frakP(Y)$ and $\frakP(YM)$ are equal. In particular, $\frakP(Y)$ is unaffected by elementary column operations of the matrix $Y$. 

After inverting $y_{1,1}$, one may perform elementary column operations to transform $Y$ into a matrix where $y_{1,1}$ is the only nonzero entry in the first row; the resulting matrix is 

\[
\widetilde{Y} \colonequals\begin{pmatrix}
y_{1,1} & 0 & 0 & \cdots & 0\\
y_{2,1} & z_{2,2} & z_{2,3} & \cdots & z_{2,n}\\
\vdots & \vdots & \vdots & & \vdots\\
y_{t,1} & z_{t,2} & z_{t,3} & \cdots & z_{t,n}\\
y_{t+1,1} & z_{t+1,2} & z_{t+1,3} & \cdots & z_{t+1,n}\\
y_{t+2,1} & z_{t+2,2} & z_{t+2,3} & \cdots & z_{t+2,n}\\
\vdots & \vdots & \vdots & & \vdots\\
y_{2t,1} & z_{2t,2} & z_{2t,3} & \cdots & z_{2t,n}
\end{pmatrix} \quad \text{where} \quad z_{i,j} = y_{i,j} - \frac{y_{i,1}y_{1,j}}{y_{1,1}}.\
\]
By construction, the ideals $\frakP(Y)$ and $\frakP(\widetilde{Y})$
are equal in the ring $S[\frac{1}{y_{1,1}}]$.  Set $Z$ to be the submatrix of $\widetilde{Y}$ obtained by
deleting the first column, and rows $1$ and $t+1$.  Note that the
nonzero entries of the matrix $\widetilde{Y}^\tr\Omega\widetilde{Y}$
are those of $Z^\tr \overline{\Omega}Z$, where $\overline{\Omega}$ is
the standard symplectic block matrix of size $2t-2$, along with the polynomials
\[
f_j\colonequals (\widetilde{Y}^\tr\Omega\widetilde{Y})_{1,j}\ =\ y_{1,1}z_{t+1,j} + (y_{2,1}z_{t+2,j}-y_{t+2,1}z_{2,j}) + \dots + (y_{t,1}z_{2t,j}-y_{2t,1}z_{t,j})
\]
for $2 \leq j \leq n$. This proves assertion $(3)$. 

Assertions $(1)$ and $(2)$ are readily verified since the entries of the matrix $Z$ do not involve the element $z_{t+1,j}$ which appears (with a unit coefficient) in $f_j$ for $2 \leq j \leq n$.
\end{proof}

\subsection{Constructing the monomial order}
\label{Subsection PfaffianMonomialOrder}

In this subsection, we describe the recipe to create a monomial order $<_B$ that creates special lead terms for the generators of the symplectic nullcone ideal. The monomial order construction is quite technical; we illustrate it with an example first:

\begin{exa}
Let $Y$ be a $4 \times 4$ matrix of indeterminates and $\KK$ be any field; set $S = \KK[Y]$. To define a monomial order $<_B$ in the polynomial ring $S$, we first define an order on the variables as follows. Sort the entries of the matrix $Y$ into blocks $B_0, B_1, B_2$, and $B_3$ as suggested in the matrix

\[
\begin{pmatrix}
1 & 3 & 1 & 0\\
2 & 2 & 0 & 0\\
0 & 1 & 3 & 1\\
0 & 0 & 2 & 2
\end{pmatrix},
\]
where the $(i,j)$-entry of the matrix is the block into which $y_{i,j}$ is sorted. Thus, for instance, $y_{1,1}$ is in the block $B_1$, $y_{1,2}$ is in $B_3$, and so on. Now, for $\gamma \in B_\ell$ and $\delta \in B_{k}$, set $\gamma < \delta$ if $\ell <k$. Then, within each set $B_\ell$, fix an arbitrary order among the variables. This gives us a total variable order in the polynomial ring $S$. Our monomial order $<_B$ is the reverse lexicographical order induced by this
variable order. 

For a polynomial $f$, let $\ini(f)$ denote the
initial monomial of $f$ with respect to our monomial order. Then
\[
\frakP = \frakP(Y) = (d_{i,j} \; \vert \; 1 \leq i<j\leq 4)
\]
is an ideal of height $5$ in $S$; the generators $d_{i,j}$ are as displayed in Remark \ref{RmkGeneratorsofNullconeIdeal}. One has
\begin{gather*}
\ini(d_{1,2}) = y_{1,1}y_{3,2},\qquad \ini(d_{2,3}) = y_{1,2}y_{3,3},\qquad \ini(d_{3,4}) = y_{1,3}y_{3,4},\\
\ini(d_{1,3}) = y_{2,1}y_{4,3},\qquad  \ini(d_{2,4}) = y_{2,2}y_{4,4},\qquad\ini(d_{1,4}) = y_{2,1}y_{4,4}.
\end{gather*}

Let $\fraka$ be the ideal of $S$ generated by the elements
$d_{1,2},d_{2,3}, d_{3,4}, d_{1,3}$, and $d_{2,4}$. Since the initial
terms of the generators of $\fraka$ are pairwise coprime, the
generators of $\fraka$ form a Gr\"{o}bner basis and it follows that
the ideal $\fraka$ is generated by a regular sequence of maximum
length in $\frakP$.

The construction of this monomial order is crucial in establishing the $F$-regularity of the symplectic nullcone $S/\frakP$, as we show next. From now on, assume that the underlying field $\KK$ has positive characteristic $p$.

Note that the polynomial
\[
f := d_{1,2}d_{2,3}d_{3,4}d_{1,3}d_{2,4}
\]
is such that 
\[
f^{p-1} \in (\fraka^{[p]}:\fraka) \minus\frakm_S^{[p]},
\]
where $\frakm_S$ denotes the homogeneous maximal ideal of $S$. This is so because its initial term $\ini(f)$ is squarefree. It follows from Corollary \ref{corMain} that $S/\frakP$ is $F$-pure. In fact, since $y_{1,4}$ does not divide $\ini(f)$, we get 
\[
y_{1,4}f^{p-1} \in y_{1,4}(\fraka^{[p]}:\fraka) \subseteq y_{1,4}(\frakP^{[p]}:\frakP) \quad \text{while}\quad  y_{1,4}f^{p-1}\notin \frakm_S^{[p]}.
\]
Since the ring $\frac{S}{\frakP}[\frac{1}{y_{1,4}}]$ is a smooth extension of the determinantal ring defined by the size two minors of a $2 \times 3$ matrix of indeterminates by Lemma \ref{lemma:matrix:invert},
it follows from Theorem \ref{theoremGlassbrenner} that the symplectic nullcone $S/\frakP$ is strongly $F$-regular. 

Finally, since $f$ lies in the ideal $\frakP^5$, it is clear that $\ini(f)$ lies in $\ini (\frakP ^5)$, and hence in $\ini (\frakP ^{(5)})$. By Theorem \ref{theorem:sqfreeinitial}, it immediately follows that the symplectic nullcone ideal $\frakP$ has a squarefree initial ideal (with respect to the monomial order $<_B$) in any characteristic.   
\end{exa}

We next construct the monomial order $<_B$ illustrated in the above
example. For ease of notation in the next Lemma, we relabel the entries of lower half of
the $2t \times n$ matrix $Y = (y_{i,j})$ as follows: Let
$w_{i,j} = y_{i+t,n-j+1}$ for $1 \leq i \leq t$ and $1 \leq j \leq
n$. Then we have

\[
Y =\begin{pmatrix}
y_{1,1} & y_{1,2} & \cdots & y_{1,n}\\
\vdots & \vdots & & \vdots\\
y_{t,1} & y_{t,2} & \cdots & y_{t,n}\\
w_{1,n} & w_{1,n-1} & \cdots & w_{1,1}\\
\vdots & \vdots & &\vdots\\
w_{t,n} & w_{t,n-1} & \cdots & w_{t,1}
\end{pmatrix}.
\]
Recall from Remark \ref{RmkGeneratorsofNullconeIdeal} that the entries of the symplectic nullcone ideal $(Y^\tr\Omega Y)$ in $\KK[Y]$ are

\[
d_{i,j} = \det \begin{pmatrix} y_{1,i} & y_{1,j}\\ w_{1,n-i+1} &
  w_{1,n-j+1} \end{pmatrix} + \cdots + \det \begin{pmatrix} y_{t,i} &
  y_{t,j}\\ w_{t,n-i+1} & w_{t,n-j+1} \end{pmatrix},
\]
for $1 \leq i<j \leq n$.

The following computation is the technical heart of this section:
\begin{lem}\label{LemmaMonomialOrder}
 Let $\KK$ be any field. There exists a monomial order $<_B$  in $\KK[Y]$ such that for all $1 \leq i < j
 \leq n$ with $j-i \leq t$, we have
 \[
 \ini(d_{i,j}) =
 y_{j-i,i}w_{j-i,n-j+1},
 \]
 where $\ini(d_{i,j})$ denotes the initial
 monomial of $d_{i,j}$ with respect to the monomial order $<_B$.
\end{lem}

\begin{proof} 
 To define a monomial order $<_B$ in the polynomial ring $S=\KK[Y]$, we first define an order on the variables as follows. Sort the entries of the matrix $Y$ into blocks $B_0,B_1,
  \dots, B_{n-1}$ according to the following formula:
  \begin{gather}\label{eqn-blocks}
  \text{$y_{i,j},w_{i,j}$ are in block $B_\ell$ where }
  \ell = \begin{cases} \textrm{(a)}\quad 2j
    +i -2 & \text{if \; } 1 \leq j< \dfrac{n-i+1}{2}, \\ \textrm{(b)}\quad 2n-2j-i &
    \text{if \; } \dfrac{n-i+1}{2} \leq j < n-i+1, \\ \textrm{(c)}\quad 0 &
    \text{otherwise.}
  \end{cases}
  \end{gather}
  Now, for $\gamma \in B_\ell$ and $\delta \in B_{k}$, set $\gamma < \delta$ if $\ell <k$. Then, within each set $B_\ell$, fix an arbitrary order among
  the variables.  This gives us a total variable order in $S$.  Our monomial order $<_B$ is the reverse lexicographical order induced by this variable order in $S$. For a polynomial $f$, let $\ini(f)$ denote the initial monomial of $f$ with respect to our monomial order. 
  
  We claim that for $i<j$ and with
  \[
  d_{i,j} = \sum_{s=1}^t y_{s,i}w_{s,n-j+1} -
  \sum_{s=1}^t y_{s,j}w_{s,n-i+1},
  \]
  we have
  $\ini(d_{i,j}) = y_{j-i,i}w_{j-i,n-j+1}$. To prove our claim, we must show that
  \[
  y_{j-i,i}w_{j-i,n-j+1} \geq y_{s,i}w_{s,n-j+1}, \quad \text{and} \quad
  y_{j-i,i}w_{j-i,n-j+1} \geq y_{s,j}w_{s,n-i+1}
  \] for all $1 \leq s
  \leq t$. As the order is reverse lexicographical, we
  must show that, for all $1 \leq s \leq t$, $y_{j-i,i}$ and
  $w_{j-i,n-j+1}$ are greater than or equal to at least one of
  $y_{s,i}$ and $w_{s,n-j+1}$, as well as at least one of $y_{s,j}$
  and $ w_{s,n-i+1}$. This can be done by analyzing all possible cases one by one. We include all the details for the sake of completeness. 
  
  Our proof will proceed as follows: First, we will show that
  $y_{j-i,i}$ and $w_{j-i,n-j+1}$ are in the same block $B_\ell$. Then we
  will show that for $s \neq j-i$ at least one of $y_{s,i}$ and
  $w_{s,n-j+1}$ is in $B_k$ for some $k<\ell$. Lastly, we will show
  that at least one of $y_{s,j}$ and $ w_{s,n-i+1}$ is in $B_k$ for
  some $k<\ell$.
  
  \textbf{Part 1}. We claim that $y_{j-i,i}$ and $w_{j-i,n-j+1}$ are
  both in $B_{i+j-2}$ if $i+j < n+1$, and are in $B_{2n-i-j}$ otherwise. 
  
  To show this, we first show that $y_{j-i,i}$ and $w_{j-i,n-j+1}$
  are not in $B_0$. Suppose $y_{j-i,i} \in B_0$. This requires that
  $n-(j-i)+1 \leq i$, implying $n+1 \leq j$, which is
  impossible. Suppose $w_{j-i,n-j+1} \in B_0$. This requires that
  $n-(j-i)+1 \leq n-j+1$, implying $i \leq 0$, which also impossible.

  Now, $i+j<n+1$ is equivalent to $i<(n-(j-i)+1)/2$ as well as to $(n-(j-i)+1)/2<n-j+1$. The former means that the block of $y_{j-i,i}$ is decided via formula (a) in \eqref{eqn-blocks} and equals $B_{2i+(j-i)-2}=B_{i+j-2}$; the latter that the block of $w_{j-i,n-j+1}$ comes from formula (b) and is $B_{2n-2(n+1-j)-(j-i)}=B_{i+j-2}$. 
  
  If on the other hand $i+j>n+1$ then $y_{j-i,i}$ follows case (b) and
  $w_{j-i,n-j+1}$ follows case (a) by the same computation. The
  corresponding blocks are computed as $B_{2n-2(i)-(j-i)}=B_{2n-i-j}$
  and $B_{2(n-j+1)+(j-i)-2}=B_{2n-j-i}$.

  In the case $i+j=n+1$, $y_{j-i,i}$ and $w_{j-i,n-j+1}$ have the same label and are in block $B_{i+j-2}=B_{2n-i-j}$.
  
  \textbf{Part 2}. Here we will show that if $y_{j-i,i}$ and
  $w_{j-i,n-j+1}$ are in $B_\ell$, then $s \neq j-i$ implies that at
  least one of $y_{s,i}$ and $w_{s,n-j+1}$ is in $B_k$ where $k<\ell$.  
  
  If $y_{s,i}$ or $w_{s,n-j+1}$ is in $B_0$, then we are done. So
  assume that $y_{s,i}$ and $w_{s,n-j+1}$ are not in $B_0$. 
  
  \begin{itemize}
  \item Suppose $i+j < n+1$ and $s < j-i$.
    
    Then $y_{j-i,i}$ and $w_{j-i,n-j+1}$ are in
    $B_{i+j-2}$. By the proof of Part 1, $ i < \dfrac{n-(j-i)+1}{2} <
    \dfrac{n-s+1}{2} $, and so $y_{s,i} \in B_{2i+s-2}$. As $s < j-i$
    implies that $2i+s-2 < i+j-2$, $y_{s,i}$ is in a lower block than $y_{j-i,i}$. 
  \item  Suppose $i+j < n+1$ and $s > j-i$.
    
    Then, $y_{j-i,i}, w_{j-i,n-j+1}\in B_{i+j-2}$. By the
    proof of Part 1,
    $ n-j+1 > \dfrac{n-(j-i)+1}{2} > \dfrac{n-s+1}{2}$, and so
    $w_{s,n-j+1} \in B_{2j-s-2}$. As $s > j-i$ implies that
    $2j-s-2 < i+j-2$, $w_{s,n-j+1}$ is in a lower block than
    $w_{j-i,n-j+1}$.
  \item Suppose $i+j \geq n+1$ and $s <j-i$.
    
    Then, $y_{j-i,i}$ and $w_{j-i,n-j+1}$ are in $B_{2n-i-j}$. By the
    proof of Part 1,
    $n-j+1 \le \dfrac{n-(j-i)+1}{2} < \dfrac{n-s+1}{2}$, and so
    $w_{s,n-j+1} \in B_{2n-2j+s}$. As $s < j-i$ implies that
    $2n-2j+s < 2n-i-j$, $w_{s,n-j+1}$ is in a lower block than
    $w_{j-i,n-j+1}$.
  \item Suppose $i+j \geq n+1$ and $s > j-i$. 
    
    Then, $y_{j-i,i}$ and $w_{j-i,n-j+1}$ are in $B_{2n-i-j}$. By the proof of Part 1, $
    i\geq \dfrac{n-(j-i)+1}{2} > \dfrac{n-s+1}{2}$, and so $y_{s,i}
    \in B_{2n-2i-s}$. As $s > j-i$ implies hat $2n-2i-s< 2n-i-j$, $y_{s,i}$ is in a lower blck than $y_{j-i,i}$.
  \end{itemize}
  
  \textbf{Part 3}. Here we will show that if
  $y_{j-i,i},w_{j-i,n-j+1}\in B_\ell$, then for $s\neq j-i$ at least one of
  $y_{s,j}$ and $ w_{s,n-i+1}$ is in $B_k$ for some $k<\ell$.
  
  If $y_{s,j}$ or $w_{s,n-i+1}$ is in $B_0$, there is nothing to show by the proof of Part 1. So
  assume that $y_{s,j}$ and $w_{s,n-i+1}$ are not in $B_0$.
  \begin{itemize}
  \item Suppose that $i+j < n+1$.
    
    By Part 1, $y_{j-i,i}$ and $w_{j-i,n-j+1}$ are in $B_{i+j-2}$.
    Since $i+j < n+1$, an inequality $n-i+1 < \frac{n-s+1}{2}$  would imply that $s < 2i-n-1 < i-j<0$, which we know to be false. Hence, $n-i+1 \geq
    \frac{n-s+1}{2}$, and thus the block of $w_{s,n-i+1}$ is decided by case (b) and equals $B_{2n-2(n-i+1)-s}=B_{2i-s-2}$. As $s > 0 > i-j$ implies that $2i-s-2< i+j-2$, $w_{s,n-i+1}$ is in a lower block than $w_{j-i,n-j+1}$.
\item Suppose that $i+j \geq n+1$.
    
    By Part 2, $y_{j-i,i}$ and $w_{j-i,n-j+1}$ are in $B_{2n-i-j}$. 
    Since $i+j \geq n+1$, an inequality $j < \frac{n-s+1}{2}$ would imply that $s <n-2j+1 \le i-j<0$, which we know to be false. Thus, $j \geq
    \frac{n-s+1}{2}$, and thus the block of $y_{s,j}$ is decided by case (b) and equals $
    B_{2n-2j-s}$. As $s > 0 > i-j$ implies that $2n-2j-s < 2n-i-j$, and so $y_{s,j}$ is in a lower block than $y_{j-i,i}$.
    \end{itemize}

    This finishes the analysis of each possible case.
\end{proof} 

Having established the monomial order $<_B$, we next exhibit a natural choice of a maximal regular sequence in the symplectic nullcone ideal.

\begin{lem}\label{LemmaRegularSequence}
 Let $\KK$ be any field and $\fraka \subseteq \frakP(Y)$ be the ideal of $\KK[Y]$ generated by the set 
 \[
 \alpha = \{d_{i,j} \;\vert \; 1 \leq i < j \leq n \text{ and } j-i \leq t\}.
 \] 
 Then the ideal $\fraka$ has the same height as $\frakP(Y)$ and $\alpha$ is a regular sequence.
\end{lem}

\begin{proof}
    First, we note that  
    \[
    \vert \alpha \vert = \begin{cases}
     \binom{n}{2} & \text{if \;} n \leq t+1,\\
     nt - \binom{t+1}{2} & \text{if \:} n \geq t,
    \end{cases}
    \] 
    where $\vert \alpha \vert$ denotes the cardinality of $\alpha$. By \cite[Theorem 6.8]{HJPS}, we have that $|\fraka| = \height(\frakP(Y))$.

    By Lemma \ref{LemmaMonomialOrder}, there exists a monomial order
    in $\KK[Y]$ such that for all $1 \leq i < j \leq n$ with $j-i \leq t$, we have
    \[
    \ini(d_{i,j}) = y_{j-i,i}w_{j-i,n-j+1}.
    \] 
    Notice that if $(i,j)\neq (k,\ell)$ then 
    $\ini(d_{i,j})$ and $\ini(d_{l,k})$ are relatively prime. Thus, the elements of $\alpha$  form a Gr\"obner basis for $\fraka$ (see, for example, \cite[Corollary 2.3.4.]{HerzogHibi11}), and the initial terms of $\alpha$ form a regular sequence.  Therefore,
    \[
    \dim(R/\fraka) = \dim(R/\ini(\fraka)) = \depth(R/\ini(\fraka)) \leq \depth(R/\fraka) \leq \dim(R/\fraka),
    \]
    so that we must have equality throughout (see \cite[Corollary 3.3.4]{HerzogHibi11}). In particular,
    \[
    \vert \alpha \vert \geq \height(\fraka) = \height(\ini(\fraka)) = \vert \alpha \vert,\] and the lemma follows from \cite[Theorem 6.8]{HJPS}.
\end{proof}

We are now ready to prove the main results of this section.

\begin{thm} \label{thm-PfaffianNullcone-Fregular} Let $Y$ be a matrix
  of indeterminates of size $2t \times n$ for positive integers $t$
  and $n$. Let $\KK$ be a field and set $S:=\KK[Y]$.
  \begin{enumerate}[\quad\rm(1)]
  \item If $\KK$ is an $F$-finite field of positive characteristic,
    the ring $S/\frakP(Y)$ is strongly
    $F$-regular.
  \item If $\KK$ has characteristic zero, $S/\frakP(Y)$ has rational singularities.
    \end{enumerate}
    
\end{thm}

\begin{proof}
  Assertion $(2)$ follows from $(1)$ since $F$-regular rings are $F$-rational and $F$-rational rings have rational singularities by \cite[Theorem 4.3]{Smith} We therefore concentrate on the case where the
  characteristic of $\KK$ is $p>0$.

  We proceed by induction on $t$. The statement is clear for $t=1$,
  since then by Remark \ref{RmkGeneratorsofNullconeIdeal}, the ideal
  $\frakP(Y)$ equals the determinantal ideal $I_2(Y)$ of the size two
  minors of $Y$. The corresponding ring is strongly $F$-regular as it
  is a Segre product of standard graded polynomial rings.

  Now assume that the assertion holds for some $t\geq 2$. By Lemma
  \ref{lemma:matrix:invert}, we have
    \[
      \frac{S}{\frakP(Y)} [\frac{1}{y_{1,n}}] \cong
      \frac{\KK[Z]}{\frakP(Z)}[y_{1,1},\dots,y_{1,n},
      y_{2,n},\dots,y_{2t,n},\frac{1}{y_{1,n}}]
    \] 
    where $Z$ is a matrix of indeterminates of size $(2t-2)\times
    n$. It follows by induction that the ring
    $\frac{S}{\frakP(Y)} [\frac{1}{y_{1,n}}]$ is strongly $F$-regular.

    In order to apply Theorem
    \ref{theoremGlassbrenner}, we must show that
    \[
      y_{1,n}(\frakP(Y)^{[p]}:\frakP(Y)) \not\subseteq \frakm_S^{[p]},
    \]
    where $\frakm_S$ is the homogeneous maximal ideal of $S$. By
    Corollary \ref{corMain}, it suffices to find an ideal $\fraka$ in
    $\frakP(Y)$ generated by a regular sequence $\alpha$ of length
    equal to the height of $\frakP(Y)$ such that
    \[
      y_{1,n}(\fraka^{[p]}:\fraka) \not\subseteq \frakm_S^{[p]}.
    \] 
    Let
    $\alpha = \{d_{i,j} \;\vert \; 1 \leq i < j \leq n \text{ and }
    j-i \leq t\}$ and $\fraka=\alpha S$. Consider the polynomial
    \[f :=  \prod_{d_{i,j}\in \alpha}d_{i,j}.\]
    Clearly
\[
y_{1,n}f^{p-1} \in y_{1,n}(\fraka^{[p]}:\fraka).
\] 
Recall that for polynomials $g$ and $h$ and a fixed monomial order $<_B$ in $S$, we
have $\ini(gh) =\ini(g)\ini(h)$. We choose the monomial order $<_B$ as
constructed in Lemma \ref{LemmaMonomialOrder}. Then we get:
\begin{eqnarray*}
    \ini(y_{1,n} f^{p-1}) &=& y_{1,n}\ini(f)^{p-1}\\
    &=& y_{1,n} \left(\prod_{\substack{1 \leq i <j\le n \\ j-i \leq t}}(y_{j-i,i}w_{j-i,n-j+1})\right)^{p-1} \notin \frakm_S^{[p]}. 
\end{eqnarray*} 
Since $\frakm_S^{[p]}$ is a monomial ideal, an element lies in
$\frakm_S^{[p]}$ if and only if each of its terms does, and so
\[
  y_{1,n}f^{p-1}\notin\frakm_S^{[p]}.
\]
We are done by Theorem \ref{theoremGlassbrenner}.
\end{proof}

\begin{cor} \label{Cor:PfaffianNullconeSqfree}
The symplectic nullcone ideal $\frakP(Y)$ has a squarefree initial ideal with respect to the monomial order $<_B$.    
\end{cor}

\begin{proof}
Let $h$ be the height of $\frakP = \frakP(Y)$ and the polynomial $f$ be as constructed in the proof of Theorem \ref{thm-PfaffianNullcone-Fregular}, i.e., \[f :=  \prod_{d_{i,j}\in \alpha}d_{i,j}.\]

Then since $f$ lies in the ideal $\frakP^h$, it is clear that $\ini(f)$ lies in $\ini (\frakP ^h)$, and hence in $\ini (\frakP ^{(h)})$. By Theorem \ref{theorem:sqfreeinitial}, it immediately follows that $\frakP$ has a squarefree initial ideal.
\end{proof}

We end this section with the following:

\begin{que}
Let $Y$ be a $2t \times n$ matrix of indeterminates for positive integers $t$ and $n$, and let $\frakP$ denote the symplectic nullcone ideal in the polynomial ring $\KK[Y]$. Denote by 
\[\calR ^S(\frakP):= \bigoplus_{k \geq 0} \frakP^{(k)} \quad \text{and} \quad G^S(\frakP):= \bigoplus_{k \geq 0} \frakP^{(k)}/\frakP^{(k+1)}\]
the \emph{symbolic Rees algebra} and the \emph{symbolic associated graded algebra} of $\frakP$ respectively. Are these rings Noetherian?
\end{que}

The proof of Theorem \ref{thm-PfaffianNullcone-Fregular} shows that the symplectic nullcone is \emph{symbolic $F$-split} (\cite[Corollary 5.10]{dSMNB21}). It immediately follows by \cite[Theorem 4.7]{dSMNB21} that the symbolic Rees algebra and the symbolic associated graded algebra of the ideal $\frakP$ are $F$-split (hence reduced). However, we do not know if either of these blowup algebras are Noetherian.

\section{The divisor class group of the symplectic nullcone}

In this section, we give a characteristic-free description of the
divisor class group and the Gorenstein property of the symplectic
nullcone. This section may be viewed as an application of the technique
of principal radical systems introduced by Hochster--Eagon in \cite{Hochster-Eagon} and
studied for the symplectic nullcone in \cite[Theorem 6.7]{HJPS}. 

\begin{thm}\label{TheoremClassGroup}
    Let $Y$ be a $2t \times n$ matrix of indeterminates for positive integers $t$ and $n$, and $\KK$ be any field. Let $R_{t,n}$ denote the symplectic nullcone $\KK[Y]/\frakP(Y)$. 
    \begin{enumerate}[\quad\rm(1)]
    
        \item If $n\geq t+1$, the divisor class group of $R_{t,n}$ is the integers. Otherwise, $R_{t,n}$ is a unique factorization domain.

         \item Let $Y|_t$ denote the submatrix of $Y$ consisting of the first $t$ columns. If $n \geq t+1$, the ideal $\frakp := I_{t}(Y|_t)$ of $R_{t,n}$ is prime of height one and generates the divisor class group. Furthermore, the canonical class of $R_{t,n}$ is given by $\frakp^{(n-t-1)}$. 

        \item $R_{t,n}$ is Gorenstein ring if and only if $n \leq t+1$.    
        \end{enumerate}
\end{thm}

Before proving the theorem, we make a crucial observation:

\begin{rmk}\label{rmk:y11_Prime}
The principal ideal $(y_{1,1})$ of $R_{t,n}$ is prime by \cite[Theorem 6.7(2)]{HJPS}. Indeed, in the notation of \cite[Theorem 6.7]{HJPS}, set $a\colonequals1$ and
\[
\sigma\colonequals \begin{cases}
(0,1,2,\dots,n-1,n) & \text{ if }\ n\le t,\\
(0,1,2,\dots,t-1,n) & \text{ if }\ n>t.\\
\end{cases}
\]
Then the ideal 
\[I_{\sigma}+J_a=I_{\sigma}+J_{s_1}=\frakP+(y_{1,1})\]
is prime in the polynomial ring $\KK[Y]$. The assertion follows.
\end{rmk}

\begin{proof}[Proof of Theorem \ref{TheoremClassGroup}]
    Let $W$ be a multiplicatively closed set of $R_{t,n}$. Recall the Nagata exact sequence \cite{Nagata} of divisor class groups   
    \[
    0 
    \longrightarrow U \longrightarrow \Cl(R_{t,n}) \longrightarrow \Cl(W^{-1}R_{t,n}) \longrightarrow 0,
    \]
    where $U$ is the subgroup of $\Cl(R_{t,n})$ consisting of the classes of pure height one ideals which have a nonempty intersection with $W$. Let $W$ be the multiplicatively closed set of $R_{t,n}$ consisting of the powers of $y_{1,1}$.
    
    For the remainder of the proof, fix $t \geq 2$. The principal ideal $y_{1,1}R_{t,n}$ is prime by Remark \ref{rmk:y11_Prime}. Thus, the only pure height one ideal containing $y_{1,1}$ is principal and $U$ is the trivial group $\{[y_{1,1}] = 0\}$. Since the class group is unaffected by a polynomial extension or by inverting a prime element, it follows from Lemma \ref{lemma:matrix:invert} (3) that the class groups $\Cl(W^{-1}R_{t,n})$ and $\Cl(R_{t-1,n-1})$ are isomorphic. This gives us an explicit inductive isomorphism of class groups
    
    \begin{gather}\label{eqn-Cl-reduction}
      \Cl(R_{t,n}) \cong \Cl(R_{t-1,n-1}) \text{ with } [\frakq]
      \mapsto [W^{-1}\frakq].
    \end{gather}

    Therefore, we get
    \begin{gather}\label{eqn-possible-Cl}
        \Cl(R_{t,n}) \cong  \begin{cases}
        \Cl(R_{1,n-t+1}) & \text{if \;} n>t+1, \\
        \Cl(R_{1,2})   & \text{if \;} n=t+1,\\
        \Cl(R_{t-n+1,1}) & \text{if \;} n<t+1.
      \end{cases}
      \end{gather}
  
    Note that for $n>t+1$, $R_{1,n-t+1}$ is the non-Gorenstein
    determinantal ring $\KK[Y_{2 \times (n-t+1)}]/I_2(Y)$ by Remark \ref{RmkGeneratorsofNullconeIdeal} (see \cite[Corollary 8.9]{BrunsVetter}), while $R_{1,2}$ is the determinantal hypersurface given by $\KK[y_{1,1},y_{1,2},y_{2,1},y_{2,2}]/(y_{1,1}y_{2,2}-y_{1,2}y_{2,1})$, and $R_{t-n+1,1}=\KK[y_{1,1},y_{2,1},\ldots,y_{t.n+1,1}]$ is a regular ring. Since the class group of the generic determinantal ring is the integers $\ZZ$ (this can also be computed directly for size two minors using the K\"{u}nneth formula for local cohomology modules \cite[Theorem
    4.1.5]{GotoWatanabe}), assertion $(1)$ follows immediately.  

    To prove assertion $(3)$,  recall that the Cohen--Macaulay ring $R_{t,n}$ is Gorenstein if and only if its canonical class is trivial. As the canonical module localizes, the canonical class of $R_{t,n}$ maps to that of $R_{t-1,n-1}$ in Display \eqref{eqn-Cl-reduction}. Obviously, the non-Gorenstein ring $R_{1,n-t+1}$ has a nontrivial canonical class while the Gorenstein ring $R_{1,2}$ and the regular ring $R_{t-n+1,1}$ have trivial canonical classes. Therefore assertion $(3)$ follows immediately from Display \ref{eqn-possible-Cl}.   

    We now prove assertion $(2)$, so assume $n\geq t+1$. Let the $(2t-2)\times (n-1)$ matrix $Z$
    and the elements $f_2, \ldots, f_n$ be as in Lemma
    \ref{lemma:matrix:invert}. We have an isomorphism of rings
    
    \begin{align*}
      \KK[Y][\frac{1}{y_{1,1}}] &\stackrel{\simeq}{\to} \KK[Z][y_{1,1}, \ldots, y_{1,n},y_{2,1}, \ldots, y_{2t,1}, f_2, \ldots, f_n, \frac{1}{y_{1,1}}]\\
      \intertext{induced by elementary column operations of the matrix $Y$ that sends}
      \frakP(Y) + I_t(Y\vert_t) &\mapsto \frakP(Z) + I_{t-1}(Z\vert_{t-1}) +(f_2, \ldots, f_n).
    \end{align*}
 By Lemma \ref{lemma:matrix:invert} (3), this map descends to the isomorphism
    \begin{align*}
      R_{t,n}[\frac{1}{y_{1,1}}] &\stackrel{\simeq}{\to} R_{t-1,n-1}[y_{1,1}, \ldots, y_{1,n},y_{2,1}, \ldots, y_{2t,1}, \frac{1}{y_{1,1}}]\\
      \intertext{with}
      \frakp := I_t(Y\vert_t) &\mapsto I_{t-1}{(Z\vert_{t-1}}). 
    \end{align*}
        
    Put $t=2$ in the above isomorphism. Clearly the ideal  $\frakq := I_1(Z\vert_1)=(z_{1,1},z_{2,1})$ is prime of height one in the ring $R_{1,n-1}$ 
    defined by the size two minors of the $2 \times (n-1)$ generic
    matrix $Z$. Further, $\frakq$ generates the class group of $R_{1,n-1}$ and the
    canonical class is given by $\frakq^{{(n-3})}$ (this can be
    computed directly from the fact that $R_{1,n-1}$ is a Segre
    product of polynomial rings; alternatively see
    \cite[Theorem 8.8]{BrunsVetter}). Each of the claims in assertion $(2)$ now follow by induction along the above isomorphism since, again, the canonical class of $R_{t,n}$ maps to that of $R_{t-1,n-1}$.
\end{proof}

\begin{rmk}
  Let $Y$ be a $2t \times n$ matrix of indeterminates for positive
  integers $t$ and $n$. It follows from Theorem
  \ref{TheoremClassGroup} and \cite[Theorem 6.3]{HJPS} that the
  symplectic nullcone $\KK[Y]/\frakP(Y)$ is Gorenstein if and only if it
  is $\QQ$-Gorenstein if and only if it is a complete intersection
  ring if and only if $n \leq t+1$. Further, note that $n=t+1$ gives
  us the only Gorenstein symplectic nullcone which is not a unique
  factorization domain, and that $n=t+2$ gives us the only (non-UFD) symplectic
  nullcone whose class group is generated by the canonical class.
\end{rmk}

\section{Compatible $F$-splittings of the varieties of complexes,\\ and the generic nullcone}

In this section, we prove that the generic nullcone is $F$-pure
and that each of its irreducible components are $F$-regular. We begin with recalling some known facts about the generic nullcone and the varieties of complexes.

\subsection{Generalities}\label{sec-splitting-generalities}
For positive integers $m$, $t$, and $n$, let $Y$ and $Z$ be $m\times t$ and $t\times n$ matrices of indeterminates respectively. Set $S$ to be the polynomial ring $\KK[Y,Z]$, and take $R$ to be the $\KK$-subalgebra generated by the entries of the product matrix $YZ$. There is a natural $\KK$-algebra isomorphism
\[R := \KK[YZ] \cong \KK[X]/I_{t+1}(X)\]
where $X$ is an $m \times n$ matrix of indeterminates, and $I_{t+1}(X)$ is the ideal generated by its size $t+1$ minors. This isomorphism is induced by mapping the entries of the matrix $X$ to the corresponding entries of the matrix $YZ$. The general linear group $\GL_t(\KK)$ acts~$\KK$-linearly on $S$ via
\[
M\colon\begin{cases} Y & \mapsto YM^{-1}\\ Z & \mapsto MZ\end{cases}
\]
where $M\in\GL_t(\KK)$. When the field $\KK$ is infinite, $R$ is precisely the ring of invariants, see \cite[\S 3]{DeConciniProcesi76} or \cite[Theorem 4.1]{Hashimoto05}.
The nullcone of the natural embedding
\[R = \KK[YZ] \subseteq \KK[Y,Z] = S\] is the typically
non-equidimensional ring $\KK[Y,Z]/(YZ)$. In this article, we refer to
this ring as the (natural) \emph{generic nullcone}. When $\KK$
has characteristic zero, the group $\GL_t(\KK)$ is linearly reductive
and thus the determinantal ring $R$ splits from $S$ as an
$R$-module. When $\KK$ has positive characteristic, this embedding
typically does \emph{not} split by \cite[Theorem 1.1]{HJPS}. This is
due to the Cohen--Macaulay property of the minimal prime ideals of the
nullcone ideal (\cite[Theorem 6.2]{Huneke}), in conjunction with the
flatness of Frobenius. We next discuss the primary decomposition of
the generic nullcone ideal.

Observe that any point $P$ in the zero set of the ideal $(YZ)S$ in
$\KK^{mt+tn}$ may be regarded as an ordered pair of matrices
$P = (A,B)$, where $A$ and $B$ are points in 
$\KK^{mt}$ and $\KK^{tn}$ respectively, such that the product matrix
$AB$ is zero. Let $r$ and $s$ denote the ranks of the matrices $A$ and
$B$ respectively. A point of the zero set of $(YZ)S$ maybe regarded
as a complex
\[
  \KK^n \xlongrightarrow[\text{}]{B} \KK^t
  \xlongrightarrow[\text{}]{A} \KK^m.
\] Consequently, we must have that $r+s \leq t$. Consider the ideals 
\[
\frakp_{r,s} = \frakp_{r,s}(Y,Z):= I_{r+1}(Y)+I_{s+1}(Z)+\YZ
\]
of $S$, where $r \leq \min \{m,t\}$, $s \leq \min \{t,n \}$, and
$r+s \leq t$. These are precisely the ideals defining the varieties of
complexes (of length two) introduced by Buchsbaum--Eisenbud in
\cite{Buchsbaum-Eisenbud}.
Notice that 
\[ 
\frakp_{r,s} \subseteq  \frakp_{r',s'} \quad \text{ whenever } \quad r' \leq r, s' \leq s,
\]
while, by \cite[Lemma
2.3]{DeConciniStrickland}, 
\[
\height(\frakp_{r,s})=(m-r)(t-r)+(n-s)(t-s)+rs.
\]
The
above discussion gives a one-to-one correspondence between the
closed $\KK$-points of the zero set of  $(YZ)S$ with the closed $\KK$-points of the varieties of
complexes of length two. Over an algebraic closure of $\KK$, the Nullstellensatz implies that the radical of the generic nullcone is $\bigcap_{r+s=t} \frakp_{r,s}(Y,Z)$. But $(YZ)S$ is radical by 
\cite{Musili-Sheshadri} (see also \cite[Theorem
4.1]{Mehta-Trivedi2}). The resulting equality of ideals over the closure of $\KK$ implies the same for $\KK$ since all ideals in question are defined over $\ZZ$:
\[
  (YZ)S = \bigcap_{r+s=t} \frakp_{r,s}(Y,Z).
\]

The varieties of complexes are Cohen--Macaulay normal domains in any
characteristic by \cite[Theorems 6.2, 7.1]{Huneke} and \cite[Theorem
2.7]{DeConciniStrickland}. Their divisor class groups and Gorenstein
property are determined independently in \cite{Bruns2} and
\cite{Yoshino}. Kempf showed that the varieties of complexes
$S/\frakp_{r,s}$ have rational singularities in characteristic zero
(\cite{Kempf2}, \cite{Kempf1}). In \cite{MehtaTrivedi} it is shown
that in positive characteristic, they are $F$-rational relative to the
resolution given by Kempf, and that they are also $F$-split.

\subsection{The localization property}

We start with a localization property for the varieties of complexes analogous to that of the symplectic nullcone in Lemma \ref{lemma:matrix:invert}.

\begin{lem}\label{lemma:localization of VoC}
  Let $Y=(y_{i,j})$ and $Z=(z_{i,j})$ be matrices of indeterminates of
  sizes $m \times t$ and $t \times n$ respectively with $m\geq 2$, $t\geq 2$, and $n\geq 2$; set
  $S := \ZZ[Y,Z]$. Let $Z'$ be the submatrix of $Z$ obtained by
  deleting the first row. Then there exists a matrix $Y'$ of size
  $(m-1) \times (t-1)$ and elements $f_1, \ldots f_n$ in
  $S[\frac{1}{y_{1,1}}]$ such that:
    \begin{enumerate}[\quad\rm(1)]
         \item The entries of $Y'$, the entries of $Z'$, and the elements $f_1, \ldots, f_n$ taken together are algebraically independent over $\ZZ$; 

         \item Along with $y_{1,1}$ and $y_{1,1}^{-1}$, the above elements generate $S[\frac{1}{y_{1,1}}]$ as a $\ZZ$-algebra;
         
         \item With $S' := \ZZ[Y',Z']$, the ideal
           $\frakp_{r,s}(Y,Z)S[\frac{1}{y_{1,1}}]$ equals
           $\frakp_{r-1,s}(Y',Z')S[\frac{1}{y_{1,1}}] +
           (f_1, \ldots, f_n)S[\frac{1}{y_{1,1}}]$, and we have
           an isomorphism
           \[\frac{S}{\frakp_{r,s}(Y,Z)}[\frac{1}{y_{1,1}}] \cong
             \frac{S'}{\frakp_{r-1,s}(Y',Z')}[y_{1,1}, \ldots,
             y_{m,1},y_{1,2}, \ldots, y_{1,t}, \frac{1}{y_{1,1}}]. \]
        
    \end{enumerate}
\end{lem}
\begin{proof}
  Let us map the entries of the matrix $Y$ to the corresponding
  entries of $MY$, where $M$ is a matrix with $m$ columns. Clearly,
  the ideal $(MY)$ generated by the entries of the matrix $MY$ is
  contained in $(Y)$. This maps the ideal $I_{i+1}(Y)$ to
  $I_{i+1}(MY)$ and the ideal $(YZ)$ to $(MYZ)$. It follows that if
  $M$ is invertible then the ideals $\frakp_{r,s}(Y,Z)$ and
  $\frakp_{r,s}(MY,Z)$ are equal. In particular, the ideal
  $\frakp_{r,s}(Y,Z)$ is unaffected by elementary row operations of
  the matrix $Y$.

After inverting $y_{1,1}$, one may perform elementary row operations to transform $Y$ into a matrix where $y_{1,1}$ is the only nonzero entry in the first column; the resulting matrix then is 

\[\widetilde{Y} = \begin{pmatrix}
    y_{1,1} & y_{1,2} & \cdots & y_{1,t} \\
    0      & y'_{2,2} & \cdots & y'_{2,t}\\
    \vdots & \vdots  &  & \vdots\\
    0 & y'_{m,2} & \cdots & y'_{m,t}
\end{pmatrix} \quad \text{where} \quad y'_{i,j} = y_{i,j} - \frac{y_{i,1}y_{1,j}}{y_{1,1}}. \]

Let $Y'$ be the submatrix of $\widetilde{Y}$ obtained by deleting the first row and column. Note that the ideal $I_{r+1}(Y)S[\frac{1}{y_{1,1}}]$ is generated by the size $r+1$ minors of the matrix $\widetilde{Y}$, and hence equals $I_r(Y')S[\frac{1}{y_{1,1}}]$. As discussed above, in the ring $S[\frac{1}{y_{1,1}}]$, the ideals $(YZ)$ and $(\widetilde{Y}Z)$ are equal. Let $Z'$ be the submatrix of $Z$ obtained by deleting the first row. Note that the entries of the matrix

\[
  \widetilde{Y}Z = \begin{pmatrix}
    y_{1,1} & y_{1,2} & \cdots & y_{1,t} \\
    0      & y'_{2,2} & \cdots & y'_{2,t}\\
    \vdots & \vdots  &  & \vdots\\
    0 & y'_{m,2} & \cdots & y'_{m,t}
\end{pmatrix} \begin{pmatrix}
    z_{1,1} & z_{1,2} & \cdots & z_{1,n} \\
    z_{2,1} & z_{2,2} & \cdots & z_{2,n}\\
    \vdots & \vdots  &  & \vdots\\
    z_{t,1} & z_{t,2} & \cdots & z_{t,n}
  \end{pmatrix}
\] 
are exactly those of the matrix $Y'Z'$ along with the elements
$f_1, \ldots f_n$, where
\[
  f_i := y_{1,1}z_{1,i} + \sum_{j=2}^t y_{1,j}z_{j,i}
\]
is the dot product of the first row of $\widetilde{Y}$ with the $i$-th column of $Z$. Thus, in the ring $S[\frac{1}{y_{1,1}}]$, we have 
\[
  (YZ) = (\widetilde{Y}Z) = (Y'Z') + (f_1, \ldots, f_n),
\]  
and the matrix $Z$ can be rewritten as

\[
  Z = \begin{pmatrix}
    \frac{f_1}{y_{1,1}} - \sum_{j=2}^t \frac{y_{1,j}z_{j,1}}{y_{1,1}} & \cdots & \frac{f_n}{y_{1,1}} - \sum_{j=2}^t \frac{y_{1,j}z_{j,n}}{y_{1,1}} \\
    z_{2,1} & \cdots & z_{2,n}\\
    \vdots   & & \vdots\\
    z_{t,1} & \cdots & z_{t,n}
  \end{pmatrix}.
\]

By the additivity of the determinant in any fixed row, a minor of $Z$
of size $s+1$ is exactly the corresponding minor of

\[ \begin{pmatrix}
    \frac{f_1}{y_{1,1}} & \cdots & \frac{f_n}{y_{1,1}} \\
    z_{2,1} & \cdots & z_{2,n}\\
    \vdots   & & \vdots\\
    z_{t,1} & \cdots & z_{t,n}
\end{pmatrix}. \]

Therefore, we get 

\[\frakp_{r,s}(Y,Z)S[\frac{1}{y_{1,1}}] = \frakp_{r,s}(\widetilde{Y},Z)S[\frac{1}{y_{1,1}}] = \frakp_{r-1,s}(Y',Z')S[\frac{1}{y_{1,1}}] + (f_1, \ldots, f_n)S[\frac{1}{y_{1,1}}].\] 

The second part of assertion $(3)$ immediately follows. Assertions $(1)$ and $(2)$ are readily verified since the matrices $Y'$ and $Z'$ do not involve the elements $z_{1,j}$ which appear (with unit coefficients) in $f_j$ for $1 \leq j \leq n$.  
\end{proof}

\subsection{Constructing the monomial order}\label{Subsection VoCMonomialOrder}

In this subsection, we describe a recipe for a monomial order $<_B$ in the polynomial ring $S$ that creates special lead terms
for the generators of the varieties
of complexes $\frakp_{r,s}$. The construction of this monomial order is quite
technical; we illustrate it with an example first. 
\begin{ntn}
We use the following notation: Let $A$ be an $m\times n$ matrix, and $i,j, k,\ell$ be integers such
that $1 \le i \le j \le m$ and $1\le k \le \ell \le n$. We use
$A^{[i,j]}_{[k,\ell]}$ to denote the $(j-i+1)\times (\ell-k+1)$ submatrix of $A$ with row
indices $i,\dots,j$ and column indices $k,\dots,\ell$.
\end{ntn}

\begin{exa}
  Let $Y$ and $Z$ be matrices of indeterminates of sizes $5 \times 3$
  and $3 \times 5$ respectively and let $\KK$ be any field; set
  $S := \KK[Y,Z]$. To define a monomial order $<_B$ in the polynomial
  ring $S$, we first define an order on the variables of $S$. Sort the entries of the matrices $Y$ and $Z$ into blocks
  $B_1, B_2,\ldots , B_{15}$ as displayed in the respective matrices

\[
 \begin{pmatrix}
12 & 11 & 10 \\
9 & 14 & 13 \\
6 & 8 & 15 \\
3 & 5 & 7 \\
1 & 2 & 4
\end{pmatrix},
\begin{pmatrix} 12 & 9 & 6 & 3 & 1 \\
14 & 11 & 8 & 5 & 2 \\
15 & 13 & 10 & 7 & 4
\end{pmatrix}.
\]

Thus, for instance, $y_{1,1}$ is in the block $B_{12}$, $y_{1,2}$ is
in $B_{11}$, $z_{1,1}$ is in the block $B_{12}$, and so on. Now, for
$\gamma \in B_\ell$ and $\delta \in B_{k}$, set $\gamma < \delta$ if
$\ell <k$. Then, within each set $B_\ell$, fix an arbitrary order
among the variables. This gives us a total variable order in $S$. Our
monomial order $<_B$ is the reverse lexicographical order induced by
this variable order in $S$.

For a polynomial $f$, let $\ini(f)$ denote the initial monomial of $f$
with respect to our monomial order. Set $c_{i,j}:=(YZ)_{i,j}$ and let
$\alpha$ be the set consisting of the elements $c_{i,j}$ in $YZ$ with
$i+j\le 4$, together with the following minors of $Y$ and $Z$:
\[
     \det\Big(Y^{[4,5]}_{[1,2]}\Big),\quad
     \det\Big(Y^{[3,5]}_{[1,3]}\Big),\quad
     \det\Big(Y^{[2,4]}_{[1,3]}\Big),\qquad
     \det\Big(Z^{[1,2]}_{[4,5]}\Big),\quad
     \det\Big(Z^{[1,3]}_{[3,5]}\Big),\quad
     \det\Big(Z^{[1,3]}_{[2,4]}\Big).
\]
Notice that 
\begin{gather*}
\ini(c_{1,1}) = y_{1,1}z_{1,1}, \qquad \ini(c_{1,2}) = y_{1,2}z_{2,2}, \qquad \ini(c_{1,3}) = y_{1,3}z_{3,3}, \\ 
\ini(c_{2,1}) = y_{2,2}z_{2,1}, \qquad \ini(c_{2,2}) = y_{2,3}z_{3,2},
\qquad \ini(c_{3,1}) = y_{3,3}z_{3,1},
\end{gather*}
and
\begin{gather*}
\ini(\det\Big(Y^{[4,5]}_{[1,2]}\Big)) = y_{4,1}y_{5,2},\quad \ini(\det\Big(Y^{[3,5]}_{[1,3]}\Big)) = y_{3,1}y_{4,2}y_{5,3},\quad \ini(\det\Big(Y^{[2,4]}_{[1,3]}\Big)) = y_{2,1}y_{3,2}y_{4,3},\\
\ini(\det\Big(Z^{[1,2]}_{[4,5]}\Big)) = z_{1,4}z_{2,5},\quad \ini(\det\Big(Z^{[1,3]}_{[3,5]}\Big)) = z_{1,3}z_{2,4}z_{3,5},\quad \ini(\det\Big(Z^{[1,3]}_{[2,4]}\Big)) = z_{1,2}z_{2,3}z_{3,4}.
\end{gather*} 
Let $f$ be the product of the
elements of $\alpha$.
Since the initial terms of the elements in $\alpha$ are
squarefree and pairwise coprime,
 $f$ has a squarefree initial term. Moreover,
 $y_{5,1}$ and $z_{1,5}$ do not divide $f$.

The construction of this monomial order is crucial in establishing the $F$-regularity of the variety of complex $S/\frakp_{r,s}(Y,Z)$, as well as the $F$-purity of the nullcone $S/(YZ)S$, as we show next. From now on, assume that the underlying field $\KK$ has positive characteristic $p$.

We begin with the case of the varieties of complexes
$S/\frakp_{r,s}$. If $r=0$ or $s=0$, then $S/\frakp_{r,s}$ is
a determinantal ring, and thus we may focus on the case where
$(r,s) = (1,2)$ or $(r,s) = (2,1)$.

Let $h$ be the height of $\frakp_{r,s}$; 
we shall show in the proof of Theorem \ref{thm-Complex-Fregular} that
$f$ lies in the symbolic power $\frakp_{r,s}^{(h)}$ whenever
$r,s\neq 0$; hence we also have
$f^{p-1}\in\frakp_{r,s}^{(h(p-1))}$ for those $r,s$. As the initial
term $\ini(f)$ is squarefree, it follows from Corollary
\ref{CorollarySymbolic} that the ideal $\frakp_{r,s}$ defines an
$F$-pure ring for $r,s\neq 0$. In fact, since $y_{5,1}$ does not divide
$\ini(f)$, we get
\[
  y_{5,1}f^{p-1} \in y_{5,1}(\frakp_{r,s}^{[p]}:\frakp_{r,s}) \quad
  \text{while} \quad y_{5,1}f^{p-1} \notin \frakm_S^{[p]}.
\]
The ring $\frac{S}{\frakp_{r,s}}[\frac{1}{y_{5,1}}]$ is a smooth
extension of the determinantal ring defined by the size two minors of
a $2 \times 5$ matrix of indeterminates by Lemma
\ref{lemma:localization of VoC}, therefore it is $F$-regular. It
follows from Theorem \ref{theoremGlassbrenner} that the varieties of
complexes $S/\frakp_{r,s}$ are $F$-regular.

Now consider the nullcone $S/(YZ)S$. As the nullcone ideal $(YZ)S$ is
radical, with minimal primes $\frakp_{r,s}$ with $r+s=3$, we have
\[(YZ)S = \bigcap_{r+s=3}\frakp_{r,s}.\] Thus, by Lemma
\ref{lem-colon-contain} and Theorem \ref{theoremFedder}, in order to establish $F$-purity, it suffices
to find, for $r+s=3$, a polynomial
$g \in \frakp_{r,s}^{[p]}:\frakp_{r,s}$ 
such that $g \notin \frakm_S^{[p]}$. Let \[g = (y_{5,1}z_{1,5}f)^{p-1}.\] We will show in the proof of
Theorem \ref{thm-length-2-F-purity} that
$ g\in\frakp_{r,s}^{[p]}:\frakp_{r,s}$ for \emph{all} $r+s = 3$. As $y_{5,1}$ and $z_{1,5}$ do not divide $f$, we get $g \notin \frakm_S^{[p]}$, and so we are done by Theorem \ref{theoremFedder}.
\end{exa}

We now construct the monomial order $<_B$ illustrated in the above
example. We start with attaching a natural number to each entry of the
matrices $Y$ and $Z$. For $Z$, this process uses exactly the numbers 
$1,\ldots,t\times n$, starts with 1 in the upper right corner of $Z$, and
proceeds along upwards oriented diagonals, in ascending order.
For $Y$, on and below the main diagonal, and starting at the lower
left corner, proceed along upwards oriented diagonals in ascending
order. Above the main diagonal of $Y$, associate to $y_{i,j}$ with $i\le j$ the block that belongs to $z_{j,j-i+1}$. In particular, only if $m=n$ will the blocks used for $Y$ be exactly $1,\ldots,m\times t$; otherwise there could be repetitions as well as omissions.


The explicit formul\ae\ for the general block numbers are as follows:
For $1\le i\le m$ and $1\le j\le n$, the variable $y_{i,j}$ is in block $B_\ell$, where

\begin{eqnarray*}
  &&\ell =\left\{
    \begin{array}{ccc}
    \binom{m-i+j+1}{2}-j+1 &\text{ if }
      &i\geq m-t+1, j\le i-m+t,\\
      \binom{t+1}{2}+(t-j)+t(m-t+i-j-1)+1&\text{ if }
      &i-m+t+1\le j\le i-1,\\
      t\cdot n+t-j-\binom{t-i+2}{2}+1&\text{ if
                                                             }
      &1\le i\le j\le t.
    \end{array}\right.\\
\end{eqnarray*}
The variable $z_{i,j}$ is in block $B_\ell$, where

\begin{eqnarray*}
  &\ell =\left\{
    \begin{array}{ccc}
      i-2j+2n+\binom{i-j+n-1}{2}&\text{ if }
                                           &1-n\le i-j\le t-n-1,\\
      \binom{t}{2}+t(i-j+n-t+1)-i+1&\text{ if }
                                           &t-n\le i-j\le
                                             -1,\\
      t\cdot n-\binom{t-i+j+1}{2}+t-i+1&\text{ if }
                                           &0\le i-j\le t-1.
    \end{array}\right.
\end{eqnarray*}
              
As in the example, our monomial order $<_B$ is the reverse lexicographical order induced by this variable order in $S$. For a polynomial $f$, let $\ini(f)$ denote the initial monomial of $f$ with respect to our monomial order.

We now consider lead terms of certain elements of the nullcone ideal:

\begin{lem}\label{length-2-coprime}
  Let $\alpha$ be the set consisting of the elements $c_{i,j}:=(YZ)_{i,j}$ with $i+j\le t+1$, along with the following minors of $Y$ and $Z$:
   \begin{itemize}
       \item $\det\Big(Y^{[m-i+1,m]}_{[1,i]}\Big)$ for $2\le i \le t-1$,
       \item $\det\Big(Y^{[i+1,t+i]}_{[1,t]}\Big)$ for $1 \le i \le m-t$,
       \item $\det\Big(Z^{[1,i]}_{[n-i+1,n]}\Big)$ for $2 \le i \le t-1 $, and
       \item $\det\Big(Z^{[1,t]}_{[i+1,t+i]}\Big)$ for $1 \le i \le n-t$.
   \end{itemize}
The initial terms of the elements of $\alpha$ are squarefree and pairwise coprime with respect to the monomial order $<_B$.
\end{lem}
\begin{proof}

As all elements considered are sums of square free monomials, their respective lead terms must also be square free. In order to show that the lead terms are also coprime, we first prove the following claim:

If
  $2\le i+j\le t+1$, then
  \begin{itemize}
  \item the entry $c_{i,j}=\sum_{k=1}^t y_{i,k}z_{k,j}$ contains a
    term for which $y_{i,k}$ and $z_{k,j}$ are both in $B_{\ell}$ for
    the same $\ell$. In fact, this happens for $k=i+j-1$;
  \item the term $y_{i,i+j-1}z_{i+j-1,j}$ is the lead term of $c_{i,j}$
    under $\le_B$.
  \end{itemize}

The first item is clear, since for $i<k$ we have defined the block
numbers of $y_{i,k}$ to agree with the block numbers of $z_{k,k-i+1}$
and so $y_{i,i+j-1}$ and $z_{i+j-1,j}$ are both in $B_{\ell}$ for the same $\ell$. 
For $i+j-1\le t$,
compare $y_{i,i+j-1}z_{i+j-1,j}$ to the other
entries $y_{i,k}z_{k,j}$ of $c_{i,j}$.  If $k>i+j-1$ then $y_{i,k}$ is
in the same row and to the right of $y_{i,i+j-1}$, which in turn is
above the main diagonal of $Y$. As the block numbers on and above the
diagonal in $Y$ decrease from left to right, $y_{i,k}$ is in
$B_{\ell'}$ for $\ell' < \ell$. Since we use revlex on the block
value, $y_{i,k}z_{k,j}<_B y_{i,i+j-1}z_{i+j-1,j}$. On the other hand,
if $k<i+j-1$, then $z_{k,j}$ is in the same column as, and above,
$z_{i+j-1,j}$. As block numbers in $Z$ increase with the row index,
$z_{k,j}$ is in $B_{\ell'}$ for $\ell' < \ell$. Since we use revlex on
the block value, $y_{i,k}z_{k,j}<_B y_{i,i+j-1}z_{i+j-1,j}$.

As the lead terms of the $c_{i,j}$ with $i+j\le t+1$ are the elements
$y_{i,i+j-1}z_{i+j-1,j}$ for that range of $i,j$, each $y_{i,j}$ with
$i \le j$ and each $z_{k,\ell}$ with $k\geq \ell$ appear exactly once in a
lead term of $c_{i,j}$.

Now observe that, in $Y$ and $Z$, the lead terms of the relevant
minors listed in Lemma \ref{length-2-coprime} are the main diagonal
terms, since any other term of this determinant involves a variable
that is under, and one variable that is above its main diagonal, and
(depending on whether one looks at minors in $Y$ or $Z$) one or the
other of these has smaller block number than any element on the main
diagonal of the minor. As the lead terms of the relevant minors of $Y$
and $Z$ are in disjoint diagonals that lie under the diagonal in $Y$,
and above the diagonal in $Z$, they are coprime to one
another and to the lead terms of the relevant $c_{i,j}$.
\end{proof}

\subsection{Main Results}

We will require the following lemma to prove the main results of this section:

\begin{lem}\label{lem-colon-contain}
    Let $I$ and $J$ be ideals in a regular ring of prime characteristic $p>0$. Then 
    \begin{itemize}
        \item $(I^{[p]}:I)\cap(J^{[p]}:J) \subseteq (I\cap J)^{[p]}:(I\cap J)$,
        \item $(I^{[p]}:I)\cap(J^{[p]}:J) \subseteq (I+J)^{[p]}:(I+J).$
    \end{itemize}
\end{lem}

\begin{proof}
  By the flatness of the Frobenius map, we have $(I\cap J)^{[p]} = I^{[p]}\cap J^{[p]}$.
  Thus we get:
    \begin{eqnarray*}
      (I\cap J)^{[p]}:(I\cap J) &=& \Big(I^{[p]}:(I\cap J)\Big)\cap \Big( J^{[p]}:(I\cap J)\Big) \\
        &\supseteq& (I^{[p]}:I)\cap(J^{[p]}:J).
    \end{eqnarray*}

    The proof of the second item is immediate (and does not require the flatness of Frobenius).
\end{proof}

We are now ready to prove:

\begin{thm} \label{thm-Complex-Fregular} Let $Y$ and $Z$ be matrices
  of indeterminates of sizes $m \times t$ and $t \times n$
  respectively for positive integers $m$, $t$, and $n$. Let $\KK$ be a field; set $S :=\KK[Y,Z]$ and suppose that $r$ and $s$ are non-negative integers with $r+s \leq t$. 
    \begin{enumerate}[\quad\rm(1)]
        \item If $\KK$ is an $F$-finite field of positive characteristic, the variety of complexes $S/\frakp_{r,s}(Y,Z)$ is strongly $F$-regular.
        \item If $\KK$ has characteristic zero, $S/\frakp_{r,s}(Y,Z)$ has rational singularities.

    \end{enumerate}
    
\end{thm}

\begin{proof}

Assertion $(2)$ follows from $(1)$ since $F$-regular rings are $F$-rational and $F$-rational rings have rational singularities by \cite[Theorem 4.3]{Smith} We therefore concentrate on the case where the
characteristic of $\KK$ is $p>0$.

We proceed by induction on $r$. The statement is clear for $r=0$, since then the ring $S/\frakp_{r,s}(Y,Z)$ is isomorphic to the determinantal ring $\KK[Z]/I_{s}(Z)$ which is strongly $F$-regular by \cite[\S 7]{HH94}.

Now assume that the assertion holds for some $r-1\geq 0$.  By suitably
restating Lemma \ref{lemma:localization of VoC}, we infer

\[
  \frac{S}{\frakp_{r,s}(Y,Z)}[\frac{1}{y_{m,1}}] \cong
  \frac{S'}{\frakp_{r-1,s}(Y',Z')}[y_{1,1}, \ldots, y_{m,1},y_{m,2},
  \ldots, y_{m,t}, \frac{1}{y_{m,1}}].
\]
where $Y'$ and $Z'$ are matrices of indeterminates of sizes
$(m-1) \times (t-1)$ and $(t-1) \times n$ respectively. It follows
from induction that the ring
$\frac{S}{\frakp_{r,s}(Y,Z)}[\frac{1}{y_{m,1}}]$ is strongly
$F$-regular.

If $s = 0$, then the ring $S/\frakp_{r,s}(Y,Z)$ is isomorphic to the
determinantal ring $\KK[Y]/I_{r}(Y)$, and we are done. So assume that
$s \geq 1$. In order to apply Theorem \ref{theoremGlassbrenner}, we must show that
    \[
      y_{m,1}(\frakp_{r,s}(Y,Z)^{[p]}:\frakp_{r,s}(Y,Z)) \not\subseteq
      \frakm_S^{[p]},
    \]
where $\frakm_S$ is the homogeneous maximal ideal of $S$. By Corollary
\ref{CorollarySymbolic}, it suffices to find a polynomial $f$
contained in $\frakp_{r,s}(Y,Z)^{(h)}$, where $h$ is the height of
$\frakp_{r,s}(Y,Z)$, such that $y_{m,1}f^{p-1} \notin
\frakm_S^{[p]}$. We will proceed to find such a polynomial; crucially this
polynomial will be \emph{the same} for all $r$ and $s$. However, before
finding this polynomial, it is useful to make the following reductions:

Firstly, since $F$-regularity is preserved by direct summands, we may assume that $t \leq \min \{m,n\}$ as the following argument shows: Choose two integers $m'\geq m$ and $n'\ge n$. Let $\overline{Y}$ be a generic $m'\times t$ matrix that contains
$Y$. Similarly, let $\overline{Z}$ be a generic $t\times n'$ matrix that
contains $Z$.  Clearly $YZ$ is a submatrix of
$\overline{Y}$ $\overline{Z}$. Consider the maps
\[
\KK[Y,Z]\to \KK[\overline{Y},\overline{Z}]\to \KK[Y,Z]
\]
where the first is the inclusion, and the second is the projection that
sends all variables in $\overline{Y}\minus Y$ and $\overline{Z}\minus Z$ to zero. The
inclusion sends $\frakp_{r,s}(Y,Z)$ into $\frakp_{r,s}(\overline{Y},\overline{Z})$ and the
projection sends $\frakp_{r,s}(\overline{Y},\overline{Z})$ onto $\frakp_{r,s}(Y,Z)$. Since
the composition is the identity map, the ring $\KK[Y,Z]/\frakp_{r,s}(Y,Z)$ is a
direct-summand of $\KK[\overline{Y},\overline{Z}]/\frakp_{r,s}(\overline{Y},\overline{Z})$. In consequence, $F$-regularity of the ring $\KK[\overline{Y},\overline{Z}]/\frakp_{r,s}(\overline{Y},\overline{Z})$ implies that
of $\KK[Y,Z]/\frakp_{r,s}(Y,Z)$.

Secondly, it suffices to consider varieties of complexes that are
exact; i.e., we may assume that $r+s=t$.  Indeed,
\[\frakp_{r,s}(Y,Z)+\frakp_{r',s'}(Y,Z) =
\frakp_{\min(r,r'),\min(s,s')}(Y,Z).\] Thus, any ideal defining a
variety of complexes may be written as the sum of ideals defining
varieties of exact complexes with the same $Y$ and $Z$.  Thus, by
Lemma \ref{lem-colon-contain}, if we can find a polynomial $f$ such that
\[
  f^{p-1} \in \frakp_{r,s}(Y,Z)^{[p]}:\frakp_{r,s}(Y,Z) \text{ and
  }f^{p-1} \in \frakp_{r',s'}(Y,Z)^{[p]}:\frakp_{r',s'}(Y,Z),
\] then
\[
  f^{p-1} \in
  \frakp_{\min(r,r'),\min(s,s')}(Y,Z)^{[p]}:\frakp_{\min(r,r'),\min(s,s')}(Y,Z).
\]
Having made these reductions, we exhibit the desired polynomial $f$ in the remainder of the proof.

Choose $f$ to be the product of the elements contained in the set
$\alpha$ as in Lemma \ref{length-2-coprime}. As the lead terms of the
elements contained in $\alpha$ are squarefree and coprime with respect to the monomial order constructed in Lemma \ref{length-2-coprime}, and as $y_{m,1}$ is not a factor of the said lead terms, $y_{m,1}f^{p-1}$ is not contained in $\frakm_S^{[p]}$. We now show that $f$ is contained in the symbolic power
$\frakp_{r,s}(Y,Z)^{(h)}$.

Recall that for $r+s=t$, the height of $\frakp_{r,s}(Y,Z)$ is $$h=(m-r)(t-r)+(n-s)(t-s)+rs=ms+nr-rs.$$

The set $\alpha$ as defined in Lemma
\ref{length-2-coprime} contains $\binom{t+1}{2}$ elements of
the form $c_{i,j}$, and so 
$f \in (YZ)^{\binom{t+1}{2}}$.

Before proceeding further, recall two well-known facts about symbolic powers: 
\begin{itemize}
\item Given an $m \times n$ matrix $A$ of indeterminates with
  $m \le n$, we have $I_{\ell+k-1}(A)\subseteq I_{\ell}(A)^{(k)}$
  whenever $1 \leq k \leq m - \ell +1$ by \cite[Proposition
  10.2]{BrunsVetter}.
\item Given prime ideals $\mathfrak{p} \subseteq \mathfrak{q}$ in
  a polynomial ring, we have
  $\mathfrak{p}^{(k)} \subseteq \mathfrak{q}^{(k)}$ for any $k\geq 0$.
\end{itemize}
%
Suppose first that $r,s>1$. Then we have
\begin{eqnarray*}
    f &\in& (YZ)^{\binom{t+1}{2}}\Bigg(\prod_{k=r+1}^{t-1}I_{k}(Y)\Bigg)I_{t}(Y)^{m-t}\Bigg(\prod_{\ell=s+1}^{t-1}I_{\ell}(Z)\Bigg)I_{t}(Z)^{n-t}\\
    &\subseteq& (YZ)^{\binom{t+1}{2}}\Bigg(\prod_{k=r+1}^{t-1}I_{r+1}(Y)^{(k-r)}\Bigg)\Big(I_{r+1}(Y)^{(t-r)}\Big)^{m-t}\Bigg(\prod_{\ell=s+1}^{t-1}I_{s+1}(Z)^{(\ell-s)}\Bigg)\Big(I_{s+1}(Z)^{(t-s)}\Big)^{n-t}\\
    &\subseteq&\frakp_{r,s}(Y,Z)^{\binom{t+1}{2}}\Big(I_{r+1}(Y)^{\big(\binom{t-r}{2}+(m-t)(t-r)\big)}\Big)\Big(I_{s+1}(Z)^{\big(\binom{t-s}{2}+(n-t)(t-s)\big)}\Big)\\
    &\subseteq& \frakp_{r,s}(Y,Z)^{\big(\binom{t+1}{2}+ \binom{t-r}{2}+(m-t)(t-r)+ \binom{t-s}{2}+(n-t)(t-s)\big)}
\end{eqnarray*}
Since $r+s=t$, an elementary computation shows that
\begin{eqnarray*}
    \binom{t+1}{2}+ \binom{t-r}{2}+(m-t)(t-r)+ \binom{t-s}{2}+(n-t)(t-s) &=& 
    ms+nr-rs = h.
\end{eqnarray*}
So,
$f$ is contained in $\frakp_{r,s}(Y,Z)^{(h)}$, as desired. 

Now assume that $r=1$, and so $h=m(t-1)+n-t+1$ while
$s+1>t-1$. Then we have
\begin{eqnarray*}
    f &\in& (YZ)^{\binom{t+1}{2}}\Bigg(\prod_{k=2}^{t-1}I_{k}(Y)\Bigg)I_{t}(Y)^{m-t}I_{t}(Z)^{n-t}\\
    &\subseteq& (YZ)^{\binom{t+1}{2}}\Bigg(\prod_{k=2}^{t-1}I_{2}(Y)^{(k-1)}\Bigg)\Big(I_{2}(Y)^{(t-1)}\Big)^{m-t}I_{t}(Z)^{n-t}\\
    &\subseteq&\frakp_{r,s}(Y,Z)^{\binom{t+1}{2}}\Big(I_{2}(Y)^{\big(\binom{t-1}{2}+(m-t)(t-1)\big)}\Big)\Big(I_{t}(Z)^{n-t}\Big)\\
    &\subseteq& \frakp_{r,s}(Y,Z)^{\big(\binom{t+1}{2}+ \binom{t-1}{2}+(m-t)(t-1)+(n-t)\big)}.
\end{eqnarray*}
Notice that
\begin{eqnarray*}
    \binom{t+1}{2}+ \binom{t-1}{2}+(m-t)(t-1)+(n-t) &=& 
    m(t-1)+n-t+1 = h.
\end{eqnarray*}
So, again, $f$ is contained in $\frakp_{r,s}(Y,Z)^{(h)}$, as desired. The
case $s=1$ is analogous to the case $r=1$, requiring only
that we switch the roles of the minors of $Y$ and $Z$. We are done
by Theorem \ref{theoremGlassbrenner}.
\end{proof}

Before proving the next theorem, we recall the notion of a \textit{compatibly split ideal}: Let $S$ be a polynomial ring over an $F$-finite field. Given a splitting $\theta: S \to S$ of the Frobenius map on $S$ (i.e. $\theta \circ F = \id_S$), an ideal $I \subset S$ is called compatibly $F$-split with respect to the splitting $\theta$ if $\theta(I) \subseteq I$. Note that $\theta$ naturally induces an $F$-splitting on $S/I$ so that $S/I$ is $F$-pure.

\begin{thm}\label{thm-length-2-F-purity}
  Let $Y$ and $Z$ be matrices of indeterminates of sizes $m \times t$
  and $t \times n$ respectively and $\KK$ an $F$-finite field of
  positive characteristic; set $S :=\KK[Y,Z]$ and assume that $r$ and $s$ are non-negative integers with $r+s \leq t$. 
  
  \begin{enumerate}[\quad\rm(1)]
      \item 
  If $t\le\min(m,n)$, the splittings of the Frobenius map on
  $S/\frakp_{r,s}(Y,Z)$ 
  may be chosen compatibly, i.e., the same splitting of the Frobenius map simultaneously splits each variety of complex. 
\item
  For any triple $(m,n,t)$, the generic nullcone $S/(YZ)S$
  is $F$-pure.
  \end{enumerate}
\end{thm}
\begin{proof}

  The ring $S/(YZ)S$ with $t>\min(m,n)$ is a direct summand of the ring $S/(\overline{Y}\overline{Z})S$, where $\overline{Y}$ is an $m' \times t$ matrix which contains $Y$ and $\overline{Z}$ is a $t \times n'$ matrix which contains $Z$ with $m'\geq \max(m,t)$ and $n'\geq \max(n,t)$. Therefore, as in the proof of Theorem  \ref{thm-Complex-Fregular}, Part (2) of the present theorem then follows from Part (1) and the fact that $F$-purity is inherited by direct summands. Furthermore, for $t\le \min(m,n)$, the $F$-purity of $S/(YZ)S$ will follow once we have shown that the $F$-splittings of each $S/\frakp_{r,s}(Y,Z)$ are compatible. So, for the remainder of the proof, assume that $t\le\min(m,n)$.

  Recall that \[(YZ) = \bigcap_{r+s=t}\frakp_{r,s}(Y,Z).\] Thus, by
  Lemma \ref{lem-colon-contain} and Theorem
  \ref{theoremFedder}, it suffices to find a (single) polynomial
  $g \in \frakp_{r,s}(Y,Z)^{[p]}:\frakp_{r,s}(Y,Z)$ for all $r+s = t$
  such that $g \notin \frakm_S^{[p]}$.
    
  Let $f$ be the product of the elements of the set $\alpha$ as in
  Definition \ref{length-2-coprime}. In the proof of
  Theorem \ref{thm-Complex-Fregular} we showed that, for $r$ and $s$
  both nonzero, $f$ is contained in $\frakp_{r,s}(Y,Z)^{(h)}$ where
  $h$ is the height of $\frakp_{r,s}(Y,Z)$, and thus
  $f^{p-1} \in \frakp_{r,s}(Y,Z)^{[p]}:\frakp_{r,s}(Y,Z)$ by Corollary
  \ref{CorollarySymbolic}.

    In order to account for the cases where $r$ or $s$ are zero, let \[g := y_{m,1}z_{1,n}f.\] Clearly $g^{p-1} \in \frakp_{r,s}(Y,Z)^{[p]}:\frakp_{r,s}(Y,Z)$ for $r$ and $s$ both nonzero. We claim that \[g \in \frakp_{r,s}(Y,Z)^{[p]}:\frakp_{r,s}(Y,Z),\] when $r$ or $s$ is zero. 
    
    Assume first that $r = 0$, and so $h = mt$. We have
    \begin{eqnarray*}
    y_{m,1}z_{1,n}f &\in& (YZ)^{\binom{t+1}{2}}\Bigg(\prod_{k=1}^{t-1}I_{k}(Y)\Bigg)I_{t}(Y)^{m-t}\\
    &\subseteq& (YZ)^{\binom{t+1}{2}}\Bigg(\prod_{k=1}^{t-1}I_{1}(Y)^{k}\Bigg)\Big(I_{1}(Y)^{t}\Big)^{m-t}\\
    &\subseteq&\frakp_{r,s}(Y,Z)^{\binom{t+1}{2}}\Big(I_{1}(Y)^{\binom{t}{2}+(m-t)t}\Big)\\
    &\subseteq& \frakp_{r,s}(Y,Z)^{\binom{t+1}{2}+ \binom{t}{2}+(m-t)t}.
\end{eqnarray*}
Notice that
\[
    \binom{t+1}{2}+ \binom{t}{2}+(m-t)t\,\, =\,\,
    mt\,\, =\,\, h.
\]
Therefore $g$ lies in
$\frakp_{r,s}(Y,Z)^{h}\subseteq \frakp_{r,s}(Y,Z)^{(h)}$, and thus
 Corollary \ref{CorollarySymbolic} implies the claim. The case $s=0$
 is analogous.
    
 We conclude the proof of the theorem by observing that the lead terms
 of the elements of $\alpha$ are squarefree and coprime by Lemma
 \ref{length-2-coprime}, and that $y_{m,1}$ and $z_{1,n}$ are not
 factors of said lead terms. Thus, $g^{p-1}$ is not contained in
 $\frakm_S^{[p]}$, and we are done by Corollary \ref{CorollarySymbolic}.
\end{proof}

\begin{cor}\label{cor:VoCsqfreeInitial}
Let $Y$ and $Z$ be matrices of indeterminates of sizes $m \times t$
  and $t \times n$ respectively and $\KK$ a field; set $S :=\KK[Y,Z]$ and assume that $r$ and $s$ are non-negative integers.   

\begin{enumerate}[\quad\rm(1)]
\item If $r+s=t$, the ideal $\frakp_{r,s}(Y,Z)$ defining the variety of exact complexes has a squarefree initial ideal.
\item If $r+s <t$ and $\KK$ has positive characteristic, the ideal $\frakp_{r,s}(Y,Z)$ defining the variety of non-exact complexes and the generic nullcone ideal $(YZ)S$ have squarefree initial ideal.
\end{enumerate}
\end{cor}

\begin{proof}
Let $<_B$ be the monomial order constructed in $\S$\ref{Subsection VoCMonomialOrder}. Choose $f$ to be the product of the elements contained in the set
$\alpha$ as in Lemma \ref{length-2-coprime}, and set
\[g:=y_{m,1}z_{1,n}f\]
as in the proof of Theorem \ref{thm-length-2-F-purity}. By the proof of Theorem \ref{thm-length-2-F-purity}, $\ini(g)$ is the product of all the variables of $S$ and lies in the initial ideal $\ini(\frakp_{r,s}^{(h)})$ for $r+s=t$, where $h$ is the height of $\frakp_{r,s}(Y,Z)$. We are done by Theorem \ref{theorem:sqfreeinitial}. For $(2)$, choose the same polynomial $g$; we are done by \cite[Theorem 3.12]{Varbaro-Koley}, where we additionally need the field to be of positive characteristic.
\end{proof}

We end with the following:

\begin{que}
Let $Y$ and $Z$ be matrices of indeterminates of sizes $m \times t$ and $t \times n$ respectively for positive integers $m$, $t$, and $n$. Let $r$ and $s$ be non-negative integers with $r+s \leq t$ and let $\frakp_{r,s}$ denote the ideal defining a variety of complexes in the polynomial ring $\KK[Y,Z]$. Denote by
\[\calR ^S(\frakp_{r,s}):= \bigoplus_{k \geq 0} \frakp_{r,s}^{(k)} \quad \text{and} \quad G^S(\frakp_{r,s}):= \bigoplus_{k \geq 0} \frakp_{r,s}^{(k)}/\frakp_{r,s}^{(k+1)}\]
the \emph{symbolic Rees algebra} and the \emph{symbolic associated graded algebra} of $\frakp_{r,s}$ respectively. Are these rings Noetherian?
\end{que} 

The proof of Theorem \ref{thm-Complex-Fregular} shows that the ideals defining the varieties of complexes are \emph{symbolic $F$-split} (see \cite[Corollary 5.10]{dSMNB21}). It immediately follows by \cite[Theorem 4.7]{dSMNB21} that the symbolic Rees algebra and the symbolic associated graded algebra of the ideal $\frakp_{r,s}$ are $F$-split (hence reduced). However, we do not know if either of these blowup algebras are Noetherian.

\section*{Acknowledgments}
We would like to thank Anurag Singh, Bernd Ulrich, Jack Jeffries, and Mel Hochster for several valuable discussions. We thank Matt Weaver for sharing useful study material with us. Vaibhav Pandey and Yevgeniya Tarasova thank Alessandro De Stefani, Jonathan Monta{\~n}o, Lisa Seccia, and Luis N\'{u}\~{n}ez-Betancourt for their advice and encouragement. Vaibhav Pandey is especially thankful to Matteo Varbaro for the invitation to the University of Genova, where a part of this work was carried out. We thank the referees for their careful review and valuable feedback.

\bibliographystyle{alpha}
\bibliography{main.bib}

\begin{thebibliography}{DSMnNnB24}

\bibitem[BE75]{Buchsbaum-Eisenbud}
David~A. Buchsbaum and David Eisenbud.
\newblock Generic free resolutions and a family of generically perfect ideals.
\newblock {\em Adv. Math.}, 18(3):245--301, 1975.

\bibitem[Bru75]{Bruns}
Winfried Bruns.
\newblock Die {D}ivisorenklassengruppe der {R}estklassenringe von
  {P}olynomringen nach {D}eterminantenidealen.
\newblock {\em Rev. Roumaine Math. Pures Appl.}, 20(10):1109--1111, 1975.

\bibitem[Bru83]{Bruns2}
Winfried Bruns.
\newblock Divisors on varieties of complexes.
\newblock {\em Math. Ann.}, 264(1):53--71, 1983.

\bibitem[BV88]{BrunsVetter}
Winfried Bruns and Udo Vetter.
\newblock {\em Determinantal rings}, volume~45 of {\em Monograf\'{\i}as de
  Matem\'{a}tica [Mathematical Monographs]}.
\newblock Instituto de Matem\'{a}tica Pura e Aplicada (IMPA), Rio de Janeiro,
  1988.

\bibitem[CGG90]{CGG}
L\'{e}andro Caniglia, Jorge~A. Guccione, and Jos\'{e}~J. Guccione.
\newblock Ideals of generic minors.
\newblock {\em Comm. Algebra}, 18(8):2633--2640, 1990.

\bibitem[CGG05]{CGG2}
M.~V. Catalisano, A.~V. Geramita, and A.~Gimigliano.
\newblock Secant varieties of {G}rassmann varieties.
\newblock {\em Proc. Amer. Math. Soc.}, 133(3):633--642, 2005.

\bibitem[CMSV18]{CMSV}
Aldo Conca, Maral Mostafazadehfard, Anurag~K. Singh, and Matteo Varbaro.
\newblock Hankel determinantal rings have rational singularities.
\newblock {\em Adv. Math.}, 335:111--129, 2018.

\bibitem[CV20]{Conca-Varbaro}
Aldo Conca and Matteo Varbaro.
\newblock Square-free {G}r\"{o}bner degenerations.
\newblock {\em Invent. Math.}, 221(3):713--730, 2020.

\bibitem[DCL81]{DeConcini-Lakshmibai}
Corrado De~Concini and Venkatraman Lakshmibai.
\newblock Arithmetic {C}ohen-{M}acaulayness and arithmetic normality for
  {S}chubert varieties.
\newblock {\em Amer. J. Math.}, 103(5):835--850, 1981.

\bibitem[DCP76]{DeConciniProcesi76}
Corrado De~Concini and Claudio Procesi.
\newblock A characteristic free approach to invariant theory.
\newblock {\em Adv. Math.}, 21(3):330--354, 1976.

\bibitem[DCS81]{DeConciniStrickland}
Corrado De~Concini and Elisabetta Strickland.
\newblock On the variety of complexes.
\newblock {\em Adv. Math.}, 41(1):57--77, 1981.

\bibitem[DK15]{Kemper-Derksen}
Harm Derksen and Gregor Kemper.
\newblock {\em Computational invariant theory}, volume 130 of {\em
  Encyclopaedia of Mathematical Sciences}.
\newblock Springer, Heidelberg, enlarged edition, 2015.

\bibitem[DSMnNnB24]{dSMNB21}
Alessandro De~Stefani, Jonathan Monta\~no, and Luis N\'u\~nez Betancourt.
\newblock Blowup algebras of determinantal ideals in prime characteristic.
\newblock {\em J. Lond. Math. Soc. (2)}, 110(2):Paper No. e12969, 50, 2024.

\bibitem[EN62]{Eagon-Northcott}
John~A. Eagon and Douglas~G. Northcott.
\newblock Ideals defined by matrices and a certain complex associated with
  them.
\newblock {\em Proc. Roy. Soc. London Ser. A}, 269:188--204, 1962.

\bibitem[Fed83]{Fedder}
Richard Fedder.
\newblock {$F$}-purity and rational singularity.
\newblock {\em Trans. Amer. Math. Soc.}, 278(2):461--480, 1983.

\bibitem[Gla96]{Glassbrenner}
Donna Glassbrenner.
\newblock Strong {$F$}-regularity in images of regular rings.
\newblock {\em Proc. Amer. Math. Soc.}, 124(2):345--353, 1996.

\bibitem[GW78]{GotoWatanabe}
Shiro Goto and Keiichi Watanabe.
\newblock On graded rings. {I}.
\newblock {\em J. Math. Soc. Japan}, 30(2):179--213, 1978.

\bibitem[Has05]{Hashimoto05}
Mitsuyasu Hashimoto.
\newblock Another proof of theorems of {D}e {C}oncini and {P}rocesi.
\newblock {\em J. Math. Kyoto Univ.}, 45(4):701--710, 2005.

\bibitem[HE71]{Hochster-Eagon}
Melvin Hochster and John~A. Eagon.
\newblock Cohen-{M}acaulay rings, invariant theory, and the generic perfection
  of determinantal loci.
\newblock {\em Amer. J. Math.}, 93:1020--1058, 1971.

\bibitem[Hes79]{Hesselink}
Wim~H. Hesselink.
\newblock Desingularizations of varieties of nullforms.
\newblock {\em Invent. Math.}, 55(2):141--163, 1979.

\bibitem[HH94]{HH94}
Melvin Hochster and Craig Huneke.
\newblock Tight closure of parameter ideals and splitting in module-finite
  extensions.
\newblock {\em J. Algebraic Geom.}, 3(4):599--670, 1994.

\bibitem[HH11]{HerzogHibi11}
J\"{u}rgen Herzog and Takayuki Hibi.
\newblock {\em Monomial ideals}, volume 260 of {\em Graduate Texts in
  Mathematics}.
\newblock Springer-Verlag London, Ltd., London, 2011.

\bibitem[Hil93]{Hilbert}
David Hilbert.
\newblock Ueber die vollen {I}nvariantensysteme.
\newblock {\em Math. Ann.}, 42(3):313--373, 1893.

\bibitem[HJPS23]{HJPS}
Melvin Hochster, Jack Jeffries, Vaibhav Pandey, and Anurag~K. Singh.
\newblock When are the natural embeddings of classical invariant rings pure?
\newblock {\em Forum Math. Sigma}, 11:Paper No. e67, 2023.

\bibitem[Hun81]{Huneke}
Craig Huneke.
\newblock The arithmetic perfection of {B}uchsbaum-{E}isenbud varieties and
  generic modules of projective dimension two.
\newblock {\em Trans. Amer. Math. Soc.}, 265(1):211--233, 1981.

\bibitem[Igu54]{Igusa}
Jun-ichi Igusa.
\newblock On the arithmetic normality of the {G}rassmann variety.
\newblock {\em Proc. Nat. Acad. Sci. U.S.A.}, 40:309--313, 1954.

\bibitem[Kem75]{Kempf1}
George~R. Kempf.
\newblock Images of homogeneous vector bundles and varieties of complexes.
\newblock {\em Bull. Amer. Math. Soc.}, 81(5):900--901, 1975.

\bibitem[Kem76]{Kempf2}
George~R. Kempf.
\newblock On the collapsing of homogeneous bundles.
\newblock {\em Invent. Math.}, 37(3):229--239, 1976.

\bibitem[Kim10]{Kim}
Sangjib Kim.
\newblock The nullcone in the multi-vector representation of the symplectic
  group and related combinatorics.
\newblock {\em J. Combin. Theory Ser. A}, 117(8):1231--1247, 2010.

\bibitem[KR15]{Rajchgot-Kinser}
Ryan Kinser and Jenna Rajchgot.
\newblock Type {$A$} quiver loci and {S}chubert varieties.
\newblock {\em J. Commut. Algebra}, 7(2):265--301, 2015.

\bibitem[KS14]{Kraft-Schwarz}
Hanspeter Kraft and Gerald~W. Schwarz.
\newblock Representations with a reduced null cone.
\newblock In {\em Symmetry: representation theory and its applications}, volume
  257 of {\em Progr. Math.}, pages 419--474. Birkh\"{a}user/Springer, New York,
  2014.

\bibitem[KV23]{Varbaro-Koley}
Mitra Koley and Matteo Varbaro.
\newblock Gr\"{o}bner deformations and {$F$}-singularities.
\newblock {\em Math. Nachr.}, 296(7):2903--2917, 2023.

\bibitem[KW06]{Kraft-Wallach}
Hanspeter Kraft and Nolan~R. Wallach.
\newblock On the nullcone of representations of reductive groups.
\newblock {\em Pacific J. Math.}, 224(1):119--139, 2006.

\bibitem[Lő23a]{Lorincz}
Andr\'as~Cristian Lőrincz.
\newblock On the collapsing of homogeneous bundles in arbitrary characteristic.
\newblock {\em Ann. Sci. \'Ec. Norm. Sup\'er. (4)}, 56(5):1313--1337, 2023.

\bibitem[Lő23b]{Lorincz23}
Andr\'as~Cristian Lőrincz.
\newblock Singularities of orthogonal and symplectic determinantal varieties.
\newblock {\em \url{https://arxiv.org/abs/2311.07549}}, 2023.

\bibitem[MP21]{Ma-Polstra}
Linquan Ma and Thomas Polstra.
\newblock {\em {$F$}-singularities: A Commutative Algebra Approach}.
\newblock preprint available at
  \url{https://www.math.purdue.edu/~ma326/F-singularitiesBook.pdf}, 2021.

\bibitem[MS83]{Musili-Sheshadri}
Chitikila Musili and Conjeerveram~S. Seshadri.
\newblock Schubert varieties and the variety of complexes.
\newblock In {\em Arithmetic and geometry, {V}ol. {II}}, volume~36 of {\em
  Progr. Math.}, pages 329--359. Birkh\"{a}user Boston, Boston, MA, 1983.

\bibitem[MT99a]{Mehta-Trivedi2}
Vikram~B. Mehta and Vijaylaxmi Trivedi.
\newblock The variety of circular complexes and {$F$}-splitting.
\newblock {\em Invent. Math.}, 137(2):449--460, 1999.

\bibitem[MT99b]{MehtaTrivedi}
Vikram~B. Mehta and Vijaylaxmi Trivedi.
\newblock Variety of complexes and {$F$}-splitting.
\newblock {\em J. Algebra}, 215(1):352--365, 1999.

\bibitem[MW21]{Makam-Wigderson}
Visu Makam and Avi Wigderson.
\newblock Singular tuples of matrices is not a null cone (and the symmetries of
  algebraic varieties.
\newblock {\em Journal für die reine und angewandte Mathematik (Crelles
  Journal).}, 2021(44):79--131, 2021.

\bibitem[Nag57]{Nagata}
Masayoshi Nagata.
\newblock A remark on the unique factorization theorem.
\newblock {\em J. Math. Soc. Japan}, 9:143--145, 1957.

\bibitem[PS73]{PS74}
Christian Peskine and Lucien Szpiro.
\newblock Dimension projective finie et cohomologie locale. {A}pplications \`a
  la d\'{e}monstration de conjectures de {M}. {A}uslander, {H}. {B}ass et {A}.
  {G}rothendieck.
\newblock {\em Inst. Hautes \'{E}tudes Sci. Publ. Math.}, (42):47--119, 1973.

\bibitem[PT24]{Pandey-Tarasova}
Vaibhav Pandey and Yevgeniya Tarasova.
\newblock Linkage and {$F$}-regularity of determinantal rings.
\newblock {\em Int. Math. Res. Not. IMRN}, (11):9323--9339, 2024.

\bibitem[Smi97]{Smith}
Karen~E. Smith.
\newblock {$F$}-rational rings have rational singularities.
\newblock {\em Amer. J. Math.}, 119(1):159--180, 1997.

\bibitem[Stu90]{Sturmfels}
Bernd Sturmfels.
\newblock Gr\"{o}bner bases and {S}tanley decompositions of determinantal
  rings.
\newblock {\em Math. Z.}, 205(1):137--144, 1990.

\bibitem[Sva74]{Svanes}
Torgny Svanes.
\newblock Coherent cohomology on {S}chubert subschemes of flag schemes and
  applications.
\newblock {\em Adv. Math.}, 14:369--453, 1974.

\bibitem[Wan01]{Wang}
Weiqiang Wang.
\newblock Resolution of singularities of null cones.
\newblock {\em Canad. Math. Bull.}, 44(4):491--503, 2001.

\bibitem[Yos84]{Yoshino}
Yuji Yoshino.
\newblock Some results on the variety of complexes.
\newblock {\em Nagoya Math. J.}, 93:39--60, 1984.

\end{thebibliography}

\end{document}